\documentclass{article}

\usepackage[margin=1in]{geometry}
\usepackage{amsmath,amsthm,amssymb,bm}
\usepackage{bbm}

\usepackage[dvipsnames]{xcolor}
\usepackage{subcaption}
\usepackage{multirow,arydshln}
\allowdisplaybreaks

\usepackage{tikz}
\usetikzlibrary{arrows,automata,positioning}
\usepackage{pgfplots}
\pgfplotsset{compat=newest}

\newtheorem{lemma}{Lemma}[section]

\newtheorem{theorem}[lemma]{Theorem}

\newtheorem{definition}[lemma]{Definition}
\newtheorem{example1}[lemma]{Example}
\newtheorem{rem1}[lemma]{Remark}
\newtheorem{assumption}[lemma]{Assumption}
\newtheorem{alg1}[lemma]{Algorithm}
\newtheorem{me1}[lemma]{Mechanism}

\newenvironment{remark}{\begin{rem1}\rm}{\end{rem1}}
\newenvironment{example}{\begin{example1}\rm}{\end{example1}}

\newenvironment{alg}{\begin{alg1}\rm}{\end{alg1}}

\newcommand{\bbr}{\mathbb{R}}
\newcommand{\R}{\mathbb{R}}
\newcommand{\N}{\mathbb{N}}
\newcommand{\bbx}{\mathbb{X}}

\newcommand{\xcal}{\mathcal{X}}

\newcommand{\vmin}{\min}

\DeclareMathOperator*{\argmin}{argMin}
\newcommand{\NE}{\operatorname{NE}}

\newcommand{\eNE}{\epsilon\NE}
\newcommand{\cl}{\operatorname{cl}}

\renewcommand{\int}{\operatorname{int}}

\DeclareMathOperator*{\argm}{arg\,min}

\DeclareMathOperator{\conv}{conv}
\DeclareMathOperator{\cone}{cone}

\newcommand{\norm}[1]{\left\Vert #1\right\Vert}

\begin{document}

\title{Approximating the set of Nash equilibria for convex games}
\author{Zachary Feinstein \thanks{Stevens Institute of Technology, School of Business, Hoboken, NJ 07030, USA, zfeinste@stevens.edu.} \and Niklas Hey\thanks{Vienna University of Economics and Business, Institute for Statistics and Mathematics, Vienna A-1020, AUT, nihey@wu.ac.at.} \and Birgit Rudloff \thanks{Vienna University of Economics and Business, Institute for Statistics and Mathematics, Vienna A-1020, AUT, brudloff@wu.ac.at.}}
\maketitle

\begin{abstract}
In~\cite{FR23}, it was shown that the set of Nash equilibria for any non-cooperative $N$ player game coincides with the set of Pareto optimal points of a certain vector optimization problem with non-convex ordering cone. To avoid dealing with a non-convex ordering cone, an equivalent characterization of the set of Nash equilibria as the intersection of the Pareto optimal points of $N$ multi-objective problems (i.e.\ with the natural ordering cone) is proven.
So far, algorithms to compute the exact set of Pareto optimal points of a multi-objective problem exist only for the class of linear problems, which reduces the possibility of finding the true set of Nash equilibria by those algorithms to linear games only. 

In this paper, we will consider the larger class of convex games.
As, typically, only approximate solutions can be computed for convex vector optimization problems, we first show, in total analogy to the result above, that the set of $\epsilon$-approximate Nash equilibria can be characterized by the intersection of $\epsilon$-approximate Pareto optimal points for $N$ convex multi-objective problems. 
Then, we propose an algorithm based on results from vector optimization and convex projections that allows for the computation of a set that, on one hand, contains the set of all true Nash equilibria, and is, on the other hand, contained in the set of $\epsilon$-approximate Nash equilibria. 
In addition to the joint convexity of the cost function for each player, this algorithm works provided the players are restricted by either shared polyhedral constraints or independent convex constraints.
\end{abstract}

\section{Introduction}
The concept of a {\em{Nash equilibrium}} is an important concept in game theory that was first introduced by John Nash in his works \cite{nash1950,nash1951}. Considering a non-cooperative game  with $N$ players ($N \geq 2$), where each player $i$ tries to minimize her individual cost function $f_i$, the Nash equilibrium describes a joint strategy of all players at which each player $i$ cannot reduce her cost $f_i$ assuming the strategies of the other players remain fixed. Hence, the Nash equilibrium provides stability in a non-cooperative game setting.

In this paper, we will mainly focus on convex games with a shared constraint set $\bbx$. 
Usually, when  considering convex games, the focus is on providing the existence of a unique equilibrium point and developing methods for finding this particular equilibrium, see e.g.~\cite{rosen65}, where additional strong convexity conditions are assumed to guarantee uniqueness of the Nash equilibrium. In this paper, we will consider convex games without additional assumptions, hence allowing for games with a unique, several, or infinitely many equilibria. Our aim is to approximate the set of Nash equilibria for any desired error bound $\epsilon>0$. 

It is well known that a Nash equilibrium does not need to be Pareto optimal and a Pareto optimal point does not need to be a Nash equilibrium (when considering the corresponding na\"ive Pareto optimization problem, see e.g.~\cite{FR23} for details). However, in~\cite{FR23}, it was shown that the set of Nash equilibria can be equivalently characterized by the set of Pareto optimal points of a specific vector optimization problem with a non-convex ordering cone, or equivalently, by the intersection of the Pareto optimal points of $N$ specific multi-objective problems. This result holds true in general for any possible non-cooperative game without making assumptions on cost functions or constraint sets. In order to use this result for numerical computations, one would need to compute the set of Pareto optimal points for multi-objective problems. So far, algorithms that provide the set of Pareto optimal points exist only for the special case of linear multi-objective optimization, see e.g.~\cite{Armand93,VanTu17,tohidi2018adjacency}. This restricts computational methods of finding the set of Nash equilibria via that characterization to linear games so far.

The goal of this paper is to introduce a method which approximates the set of all Nash equilibria of a convex game. Hence, $\epsilon$-approximate solution concepts are considered for both, Nash equilibria and Pareto optimality. Similar to the characterizations proven in \cite{FR23}, the set of $\epsilon$-Nash equilibria for any possible $N$-player game can be characterized by the intersection of $\epsilon$-Pareto optimal points of $N$ multi-objective problems for any $\epsilon>0$. For convex games these multi-objective problems are convex. In general, the set of Pareto optimal points as well as the set of $\epsilon$-Pareto optimal points are not finitely generated and thus cannot be computed exactly. However, due to the Lipschitz continuity of the convex cost functions and by making additional assumptions on the structure of the convex constraint set, we will (for each of the $N$ specific convex problems) be able to compute a finitely generated set which, on one hand, contains all Pareto optimal points and, on the other hand, is a subset of the $\epsilon$-Pareto optimal points for this problem for some specific $\epsilon>0$. As a consequence, by taking the intersections over these $N$ sets yields a finitely generated set $X$ which contains the set of true Nash equilibria $\NE(f,\mathbb X)$ for the convex game while being contained in the set of $\epsilon$-approximate Nash equilibria $\eNE(f,\mathbb X)$, i.e.
\begin{align*}
    \NE(f,\mathbb X) \subseteq X  \subseteq \eNE(f,\mathbb X).
\end{align*}
This result provides some advantages: It is guaranteed that each element of interest, namely each equilibrium point, is contained in the computed set. Furthermore, 
each element in the found set is guaranteed to be at least almost at equilibrium, i.e.\ deviating at most by $\epsilon$ from the cost provided by a true equilibrium point. 
Additionally to the Lipschitz continuity for the jointly convex cost functions, we distinguish between two different assumptions on the structure of the shared constraint set: In the first case we assume that the shared constraint set is a polytope. In the second case we consider for each player $i$ a convex constraint set that is independent from the constraints of the other players. For both cases, the proposed algorithm is proven to work correctly.  Computational examples are provided which illustrate the sandwich result.

\section{Definitions}
Let us consider the following non-cooperative shared constraint games for $N \geq 2$ players. Each player $i$ (for $i = 1,...,N$) considers strategies in the linear space $\xcal_i= \R^{n_i}$, where $n_i\in\N$, and where we do not impose any condition that $\xcal_i$ and $\xcal_j$ are equal.
Furthermore, each player $i$ has a cost function $f_i: \prod_{j = 1}^N \xcal_j \to \bbr$ she seeks to minimize.  This cost $f_i(x)$ for $x \in \prod_{j = 1}^N \xcal_j$ may depend on player $i$'s strategy $x_i \in \xcal_i$ as well as the strategy chosen by all other players $x_{-i} \in \prod_{j \neq i} \xcal_j$. The vector of cost functions is denoted by $f=(f_1,...,f_N)$ with $f: \prod_{i = 1}^N \xcal_i \to \R^N$.
Assume further a shared constraint set $\bbx \subseteq \prod_{j = 1}^N \xcal_j$ for the joint strategy $x \in \prod_{j = 1}^N \xcal_j$ of all players.  This shared constraint condition is common in the literature (see, e.g.,~\cite{rosen65,facchinei2007generalized,Kulkarni2017}). In a non-cooperative shared constraint game each player $i$ minimizes her cost function given all other players fix their strategies $x_{-i}^* \in \prod_{j \neq i} \xcal_j$. That is, for all $i = 1,...,N$ the following optimization problem is considered
\begin{equation}\label{eq:game}
x_i^* \in \argm\{f_i(x_i,x_{-i}^*) \; | \; (x_i,x_{-i}^*) \in \bbx\}.
\end{equation}
Non-cooperative shared constraint games are a special type of generalized games as not only the cost function, but also the constraint set of player $i$ in optimization problem \eqref{eq:game} can depend on the strategy $x_{-i}^*$ of all other players. In the notation of \cite{Kulkarni2017}, in which such problems are explicitly called ``generalized Nash games with shared constraints'', the constraint for player $i$ can be denoted by $K_i(x_{-i}^*) := \{x_i \in \xcal_i \; | \; (x_i,x_{-i}^*) \in \bbx\}$ so that problem \eqref{eq:game} can be written as $x_i^* \in \operatorname{argmin}\{f_i(x_i,x_{-i}^*) \; | \; x_i \in K_i(x_{-i}^*)\}$. Shared constraints can be used e.g. for system-wide constraints, for details see~\cite{Kulkarni2017}. Notably, shared constraint games encompass classical games in which the players interact only through the cost functions $f_i$. In our notation, a classical game can be encoded by the box-type shared constraint $\bbx = \prod_{j = 1}^N \bbx_i$ for $\bbx_i \subseteq \xcal_i$ for every player $i$.

\begin{definition}\label{defn:nash-shared} 
 A joint strategy $x^* \in \bbx$ satisfying \eqref{eq:game} for all $i = 1,...,N$ is called a \textbf{\emph{Nash equilibrium}}.  Thus, $x^* \in \bbx$ is a Nash equilibrium if, for any player $i$, 
$f_i(x_i,x_{-i}^*) \geq f_i(x^*)$ for all strategies $x_i \in \xcal_i$ with $(x_i,x_{-i}^*) \in \bbx$.  The set of all Nash equilibria is denoted by $\NE(f,\bbx)$.
\end{definition}
When numerical methods are used, one  has to deal with approximation errors. Therefore, we will now introduce the notion of approximate Nash equilibria. In a Nash equilibrium, no player has an incentive to change his behavior. An approximate Nash equilibrium allows the possibility that a player may have a small incentive to deviate.
They are also needed for, e.g., sensitivity analysis of Nash equilibria, see~\cite{feinstein2022}, or can be interpreted as Nash equilibria to some perturbed preferences~\cite{JRT2012}.

\begin{definition}\label{defn:nash-approx} \cite[Section~2.6.6]{nisan2007algorithmic}
Let $\epsilon>0$. The joint strategy $x^* \in \bbx$ is called an $\bm{\epsilon}$-\textbf{\emph{approximate Nash equilibrium}} for the game if, for any player $i$ it holds $f_i(x_i,x_{-i}^*) + \epsilon \geq f_i(x^*)$ for any $x_i \in \xcal_i$ such that $(x_i,x_{-i}^*) \in \bbx$.
The set of all $\epsilon$-approximate Nash equilibria will be denoted by $\eNE(f,\bbx)$.  That is, 
\[\eNE(f,\bbx) := \{x^* \in \bbx \; | \; \forall i: \; f_i(x^*)\leq\inf\{f_i(x_i,x_{-i}^*) \; | \; (x_i,x_{-i}^*) \in \bbx\} + \epsilon\}.\]
\end{definition}

Let us now introduce the notion of Pareto optimality from multi-objective optimization, which also plays a fundamental role in economics as it is related to the interpersonal incomparability of utility.
Consider a linear space $\xcal$ as well as the space $\R^m$ with the natural ordering cone $\R^m_+$. Recall that a set $C \subseteq \R^m$ is called a cone if $\alpha C\subseteq C$ for all $\alpha \geq 0$. The convex cone $\R^m_+$ introduces 
a partial order  
$\leq$ on $ \R^m$ via
$
	x\leq y \iff y-x\in \R^m_+,
$
for $x,y\in  \R^m$. 
A multi-objective optimization problem is a problem of the form 
\begin{equation}\label{eq:VOPg}
	\vmin \{g(x) \; | \; x \in \bbx\} 
\end{equation}
 for some feasible region $\bbx \subseteq \xcal$, a vector function $g: \bbx \to \R^m$ and ordering cone $\R^m_+$. 
To minimize this vector-valued function, and thus to solve the multi-objective optimization problem~\eqref{eq:VOPg}, means to compute the set of minimizers, also called Pareto optimal points, or efficient points, which are defined as follows.  
  \begin{definition}{\cite{GJN06}}\label{defn:pareto}
  An element $x^* \in \mathbb X$ is called \textbf{Pareto optimal} for problem~\eqref{eq:VOPg} if 
\[
	\big(g(x^*) - \R^m_+\setminus{\{0\}}\big)\cap g[\bbx] =\emptyset. 
\]
The set of all Pareto optimal points of problem~\eqref{eq:VOPg} is denoted by 
\[
	\argmin \{g(x) \; | \; x \in \bbx\}.
\] 
\end{definition}
We use the notation $g[\bbx]:=\{g(x) \; | \; x\in \bbx\}  \subseteq \R^m$ for the image of the feasible set $\bbx$. Note that the set $\argmin \{g(x) \; | \; x \in \bbx\}$ can equivalently be characterized as the set of feasible points $x^*$ that map to minimal elements of $g[\bbx]$. That is, $x^* \in \bbx$ is Pareto optimal if, whenever $g(x) \leq g(x^*)$ for some $x \in \bbx$, then $g(x^*) \leq g(x)$ holds, see~\cite[Def.\ 3.1(c), Def.\ 4.1(a) and Def.\ 7.1(a)]{Jahn11}.  

Denote by $\mathcal P:=\cl (g[\bbx] +\R^m_+)$ the upper image of problem~\eqref{eq:VOPg}. The boundary of the set $\mathcal P$ contains the so called efficient frontier, which is the collection of functional values $g(x^*)$ of all Pareto optimal points $x^*$ of problem~\eqref{eq:VOPg}. Numerical methods typically only compute approximations of $\mathcal P$, except for the case of linear multi-objective optimization problems, where the set $\mathcal P$ as well as the set of all Pareto optimal points can be computed exactly.
This can be seen e.g. for the case of convex problems, where the algorithms in~\cite{LRU14} result in a polyhedral inner approximation of $\mathcal P$ for a given error level $\epsilon>0$ and a fixed direction $c \in \R^m_+\setminus{\{0\}}$. This motivated the following definition of $\epsilon$-Pareto optimal points, which are feasible points $x^* \in \mathbb X$ whose image values $g(x^*)$ are only in $\epsilon$-distance (in direction $c$) to the efficient frontier.

\begin{definition}{\cite{K79}\cite[Def.~2.2]{GJN06}}
\label{def:Pareto-c-solution}
Let $c \in \R^m_+\setminus{\{0\}}$. An element $x^* \in \mathbb X$ is called $\bm{\epsilon}$-\textbf{Pareto optimal} with respect to $c$ for problem~\eqref{eq:VOPg} if 
\[
	\big(g(x^*)- \epsilon c - \R^m_+\setminus{\{0\}}\big)\cap g[\bbx] =\emptyset. 
\]

\end{definition}
Definition \ref{def:Pareto-c-solution} holds for an arbitrary choice of direction $c \in \R^m_+\setminus{\{0\}}$. However, as we will see in the following section, for a specific vector optimization problem a connection between the set of $\epsilon$-Pareto optimal points and the set $\eNE(f,\bbx)$ for some $N$-player game can be explicitly stated for the fixed direction $\Bar{c}:=(0,...,0,1)^\top \in \R^m_+\setminus{\{0\}}$. Therefore we will, for the remainder of this paper, consider the fixed direction $\Bar{c}$ and denote the set of all $\epsilon$-Pareto optimal points with respect to $\Bar{c}$ for problem~\eqref{eq:VOPg} by 
\[
	\epsilon\argmin \{g(x) \; | \; x \in \bbx\} .
\] 
Often, the notion of weakly Pareto optimal points and weakly $\epsilon$-Pareto optimal points are important.
\begin{definition}{\cite{GJN06}}
\label{def:wPareto-c-solution}
$x^* \in \mathbb X$ is called \textbf{weakly Pareto optimal} for problem~\eqref{eq:VOPg} if 
$
	\big(g(x^*) - \int\R^m_+\big)\cap g[\bbx] =\emptyset. 
$
Let $c \in \R^m_+\setminus{\{0\}}$. An element $x^* \in \mathbb X$ is called weakly $\epsilon$-Pareto optimal with respect to $c$ for problem~\eqref{eq:VOPg} if 
$
	\big(g(x^*)- \epsilon c - \int\R^m_+\big)\cap g[\bbx] =\emptyset. 
$
\end{definition}

\section{Equivalence of Nash and Pareto and their approximations}
Consider now for each player $i$ the following convex multi-objective optimization problem
\begin{equation}\label{eq:vop}
\min\{(x_{-i},-x_{-i},f_i(x)) \; | \; x \in \bbx\}.
\end{equation}
That is, we consider problem~\eqref{eq:VOPg} with objective function $g(x)=(x_{-i},-x_{-i},f_i(x))$.
Note that the ordering cone of problem~\eqref{eq:vop} is $\R^{m_i}_+$ with $m_i=2 \sum_{j=1,\\ j \neq i}^N n_j +1$. This objective space dimension can be further reduced by considering the equivalent convex multi-objective optimization problem
\begin{equation}\label{eq:vop2}
\min\left\{\left(x_{-i},-\sum_{j=1,\\ j \neq i}^N \sum_{k = 1}^{n_j} x_{jk},f_i(x)\right) \; \bigg\vert  \; x \in \bbx \right\}
\end{equation}
with natural ordering cone in dimension $\sum_{j=1,\\ j \neq i}^N n_j+2$, where $x_{jk}$ denotes the $k$-th element of the vector $x_j$, for $j\in\{1,...,N\}$ and $k\in\{1,...,n_j\}$. Hence, $s:=-\sum_{j=1,\\ j \neq i}^N \sum_{k = 1}^{n_j} x_{jk}$ is the sum over all components of the vector $-x_{-i}$. The equivalence holds as minimizing component-wise the vector $(x_{-i},-x_{-i})$ gives the same minimizers (being the set of all $x_{-i}$) as minimizing the vector $(x_{-i},s)$. This dimension reduction is important for computational aspects and should be kept in mind. For purely aesthetical reasons and since it simplifies explanations we will work with problem~\eqref{eq:vop} instead of problem~\eqref{eq:vop2} in our exposition. A further dimension reduction by using a non-standard ordering cone is discussed in Remark~\ref{rem:dim_red}.

Note that minimizing the multi-objective function $(x_{-i},-x_{-i},f_i(x))$ with respect to the natural ordering cone in~\eqref{eq:vop} fixes, for player $i$, the strategy $x_{-i}$ of the other players while minimizing his objective function $f_i(x)$. This clearly corresponds to solving the optimization problem~\eqref{eq:game} for any given strategy $x_{-i}=x_{-i}^*$. Doing that for each player, i.e.\ taking the intersection of the Pareto optimal points of each problem~\eqref{eq:vop} over all players computes the fixed points. The following theorem is immediate.
\begin{theorem}{\cite{FR23}}\label{thm:nash}
The set of Nash equilibria of any non-cooperative game~\eqref{eq:game} coincides with the intersection of the Pareto optimal points of problems~\eqref{eq:vop} over all $i\in\{1,...,N\}$, i.e.,
\begin{align}\label{eq:shared}
\NE(f,\bbx) = \bigcap_{i = 1}^N \argmin\{(x_{-i},-x_{-i},f_i(x)) \; | \; x \in \bbx\}.
\end{align}
\end{theorem}
The above theorem was given in~\cite{FR23} in a more general framework and in a different formulation using a non-convex ordering cone (see \cite[Th.~2.6, Cor.~2.8]{FR23}). We provide here the above intuitive and simplified formulation that was already mentioned in Remark~3.5 in~\cite{FR23} for the linear case, but holds of course also in general.
A short remark connecting the algorithm proposed in this paper to the problem formulation using the non-convex ordering cone is given in Remark~\ref{rem:dim_red}.

In the case of linear games the set of all Pareto optimal points of problem~\eqref{eq:vop} can be computed exactly and Theorem~\ref{thm:nash} can be used to numerically compute the set of all Nash equilibria of such games, see~\cite{FR23}. If the game is not linear, approximations need to be considered. In the following, we will therefore relate the set of $\epsilon$-approximate Nash equilibria with the $\epsilon$-Pareto optimal points of problem~\eqref{eq:vop}. To do so we will fix the directions $\bar{c}_i=(0,...,0,1)^\top\in  \R^{m_i}_+\setminus{\{0\}}$ for all $i\in\{1,...,N\}$. The choice of this direction ensures that no $\epsilon$-deviation is allowed in the other player's strategies, so for each player $i$ the strategy $x_{-i}$ of the other players stays fixed, while an $\epsilon$-deviation is allowed for the objective $f_i(x)$ of player $i$. 
\begin{theorem}\label{thm:nash-approx}
The set of $\epsilon$-approximate Nash equilibria of any non-cooperative game~\eqref{eq:game} coincides with the intersection of the $\epsilon$-Pareto optimal points of problems~\eqref{eq:vop} for direction $\bar{c}_i=(0,...,0,1)^\top$ over all $i\in\{1,...,N\}$, i.e.,
\begin{align}\label{eq:shared_eps}
\eNE(f,\bbx) = \bigcap_{i = 1}^N \epsilon\argmin\{(x_{-i},-x_{-i},f_i(x)) \; | \; x \in \bbx\}.
\end{align}
\end{theorem}
\begin{proof}
The choice of the direction $\bar{c}_i=(0,...,0,1)^\top\in  \R^{m_i}_+\setminus{\{0\}}$ and the choice of the objective function $g(x)=(x_{-i},-x_{-i},f_i(x)) $ ensures that by Definition~\ref{def:Pareto-c-solution}, $x^*\in \epsilon\argmin \{(x_{-i},-x_{-i},f_i(x))|x \in \mathbb X \}$ is equivalent to: there does not exists an $x\in \mathbb X$ such that $x_{-i} = x_{-i}^*$ and  $f_i(x)< f_i(x^*)-\epsilon$. This is equivalent to $f_i(x_i,x_{-i}^*) + \epsilon  \geq f_i(x^*)$ for any $x_i \in \xcal_i$ such that $(x_i,x_{-i}^*) \in \bbx$. Since this has to hold for all $i\in\{1,...,N\}$, the equivalence to Definition~\ref{defn:nash-approx} follows.
\end{proof} 
\begin{remark}
Theorems~\ref{thm:nash} and~\ref{thm:nash-approx} can also be stated:
\begin{itemize}
\item for generalized games with individual constraint sets $\mathbb{C}_i \subseteq \prod_{j = 1}^N \xcal_j$ for $i\in\{1,...,N\}$,  
i.e., when problem~\eqref{eq:game} in Definition~\ref{defn:nash-shared} is replaced by $x_i^* \in \argm\{f_i(x_i,x_{-i}^*) \; | \; (x_i,x_{-i}^*) \in \mathbb{C}_i\}$. Then, the constraints sets $\bbx$ in the Pareto problems in Theorems~\ref{thm:nash} and~\ref{thm:nash-approx} have to be replaced by the individual constraint sets $\mathbb{C}_i$ for $i\in\{1,...,N\}$.
\item for general linear spaces $\xcal_i$ for the strategies. Then, the ordering cone of the corresponding vector optimization problem has to be adapted accordingly, see Theorem~2.6 in~\cite{FR23}.
\item for vector games, i.e.\ when the objective functions of the players are vector functions. In this case, the ordering cone $\R^{m_i}_+$ in problem~\eqref{eq:vop} has to be adapted accordingly, see also Section~4 in~\cite{FR23}.
\end{itemize}
Within this paper, we are mainly interested in using Theorems~\ref{thm:nash} and~\ref{thm:nash-approx} to numerically approximate the set of all Nash equilibria for certain convex games as detailed in Section~\ref{sec:convex} below. Thus, we will only work with Theorems~\ref{thm:nash} and~\ref{thm:nash-approx} as stated above.
\end{remark}
\begin{remark}\label{rem:weak_opt}
Note that even though in multi-objective optimization one often works with weakly Pareto optimal points or weakly $\epsilon$-Pareto optimal points, for our particular multi-objective optimization problem~\eqref{eq:vop} the results in Theorems~\ref{thm:nash} and~\ref{thm:nash-approx} only hold for the stronger concepts of Pareto optimal points, respectively $\epsilon$-Pareto optimal points. In particular, note that every feasible point $x\in \mathbb X$ is weakly Pareto optimal (and hence also weakly $\epsilon$-Pareto optimal)
for problem~\eqref{eq:vop}. 
This follows easily from the fact that for any arbitrary $\Bar{x} \in \bbx$ it is $((\Bar{x}_{-i},-\Bar{x}_{-i})^\top -\int \mathbb R^{m_i-1}_+ )\cap \{(x_{-i},-x_{-i})\; | \; x \in \bbx\}=\emptyset$. Thus, the concepts of weakly ($\epsilon$-) Pareto optimality are not meaningful for problem~\eqref{eq:vop}. 
\end{remark}

\section{Convex games}\label{sec:convex}
The aim of this paper is to develop an algorithm that uses representations~\eqref{eq:shared} and~\eqref{eq:shared_eps} to approximate the set of Nash equilibria of a non-cooperative game.
To do so, we will focus for the remaining part of this paper on convex games satisfying the following assumption. 
\begin{assumption}{\label{ass:convex}}
\begin{enumerate}
    \item The shared constraint set $\bbx \subseteq \prod_{j=1}^N \mathcal{X}_j$ is convex and compact.
    \item Each cost function $f_i: \prod_{j=1}^N \mathcal{X}_j \to \mathbb R $ is convex.
\end{enumerate}
\end{assumption}
Under Assumption~\ref{ass:convex},~\cite[Theorem~1]{rosen65} guarantees $\NE(f,\mathbb{X}) \neq \emptyset$.
The first assumption is typical for convex games. However, in the second assumption, the joint convexity of each player's cost function is assumed; this is stronger than the usual assumption for convexity in the decision variable for player $i$ only (see, e.g.,~\cite{nikaido1955note,rosen65}). Note that both assumptions imply that each function $f_i$ is Lipschitz continuous on the compact set $\bbx$. Let us denote by $L>0$ the largest of the corresponding Lipschitz constants. Thus, we have for all $i = 1,...,N$ that $|f_i(x_1)-f_i(x_2)| \leq L \norm{x_1-x_2}$ for $x_1,x_2 \in \bbx$ where $\norm{\cdot}$ is the $L_1$ norm on $ \prod_{j=1}^N \mathcal{X}_j $. This will be useful to relate an approximation error made in the preimage space to the corresponding error in the image space using the Lipschitz constant. 

Under these convexity assumptions and Theorem~\ref{thm:nash-approx}, we can hence use convex multi-objective optimization methods to compute an approximation of the set of Nash equilibria.
The following difficulties appear
\begin{itemize}

\item[i)] In convex multi-objective optimization one usually focuses on weakly Pareto optimal points (or weakly $\epsilon$-Pareto optimal points) since they can equivalently be characterized as solutions to the weighted sum scalarization. That is, a point $x^* \in \mathbb X$ is weakly Pareto optimal for the convex problem~\eqref{eq:VOPg} if and only if  $x^*$ is a solution to the scalar problem $\min_{x\in \bbx}\sum_{i=1}^m w_i g_i(x)$ for some $w\in\R^m_+\setminus{\{0\}}$ (Corollary 5.29 of \cite{Jahn11}). However, as stated in Remark~\ref{rem:weak_opt}, the concept of weak Pareto optimality is not meaningful for our problem~\eqref{eq:vop} and we have to work with Pareto optimal points instead. For Pareto optimal points there is not a one-to-one correspondence to solutions of weighted sum scalarizations, but only the following implication: a point $x^* \in \mathbb X$ is Pareto optimal for a convex problem~\eqref{eq:VOPg} if $x^*$ is a solution to the scalar problem $\min_{x\in \bbx}\sum_{i=1}^m w_i g_i(x)$ for some $w\in\R^m_{++}$ (Theorem 5.18(b) of \cite{Jahn11}). The absence of an equivalent characterization of the set of Pareto optimal points through scalarizations will make it impossible to solve our problem in full generality for convex games. 
However, despite this issue we will be able to compute a set which contains the set of all Pareto optimal points of the convex problem \eqref{eq:vop} (and is included in the set of all $\epsilon$-Pareto optimal points) if we make additional assumptions on the structure of the constraint set $\bbx$. We will consider two different structures of $\bbx$, one in Assumption~\ref{asspolycase} and one in Remark~\ref{rem:inconstr}.

\item[ii)] To the best of our knowledge, there is no algorithm so far that computes or approximates the set of all Pareto optimal points or weakly Pareto optimal points for a convex multi-objective optimization problem. In the linear case, such algorithms exist, see~\cite{Armand93,VanTu17,tohidi2018adjacency}, and have been used to compute the set of Nash equilibria for linear games in~\cite{FR23}. For convex multi-objective optimization problems, it is often not necessary to know the set of all (weakly) Pareto optimal points, as one is usually satisfied in finding finitely many weakly Pareto optimal points that approximate the upper image $\mathcal P$. In detail: one usually computes a finite set of weakly Pareto optimal points whose images provide a polyhedral inner approximation $\mathcal{P}^{In}$ of $\mathcal P$ that is in $\varepsilon$-distance to $\mathcal P$  for a given error level $\varepsilon>0$ and a fixed direction $c \in \int \R^m_+$ in the following sense
\begin{align}{\label{innerappr1}}
    \mathcal{P}^{In} \subseteq \mathcal{P} \subseteq \mathcal{P}^{In}-\varepsilon \{{c}\},
\end{align}
see, e.g.,~\cite{LRU14}. This is not sufficient for our purposes as we need on one hand not weakly, but Pareto optimal points, and on the other hand, we need the set of all Pareto optimal points (or $\epsilon$-Pareto optimal points). However, we will see that such a polyhedral approximation~\eqref{innerappr1} will be a first step to reach our goal.

\item[iii)] The algorithm in~\cite{LRU14}, providing a polyhedral inner approximation of $\mathcal P$ satisfying~\eqref{innerappr1}, works under the assumption that the direction $c \in \int\R^m_+$. This assumption is clearly violated for our problem~\eqref{eq:vop} as we need the directions to be  $\bar{c}_i=(0,...,0,1)^\top\in  \R^{m_i}_+\setminus{\{0\}}$ to obtain the representation~\eqref{eq:shared}. 
\end{itemize}

As mentioned above in issue~(i), we try to cover the set of all Pareto optimal points and therefore make additional assumptions on the structure of the constraint set $\bbx$. In the following, we will consider constraints sets $\bbx$ that are polytopes, whereas in Remark~\ref{rem:inconstr}, we consider the case where each player $i$ has an independent convex constraint set $\bbx_i \subseteq \mathcal{X}_i$. Problem~(iii) can be handled by a small modification of the algorithm in~\cite{LRU14} that is possible because of the particular structure of the objective function of problem~\eqref{eq:vop} being linear in all but the last component and the directions $\bar{c}_i$ being zero in these linear components.

Let us now discuss problem~(ii) in detail. The reason why approximations are considered for convex problems is that it is in general not possible to generate the exact upper image by finitely many points or to obtain an analytic expression for it. 
For the same reasons, only approximations of set of Pareto optimal points $\argmin \{g(x) \; | \; x \in \bbx\}$ respectively of the set of $\epsilon$-Pareto optimal points $\epsilon\argmin \{g(x) \; | \; x \in \bbx\}$ of a convex (but not linear) problem~\eqref{eq:VOPg} will be considered. We will focus here on polyhedral approximations as they are finitely generated. 

Hence, we cannot expect to compute the sets $\NE(f,\bbx)$ or $\eNE(f,\bbx)$ via~\eqref{eq:shared}, respectively~\eqref{eq:shared_eps}, exactly. However, below we propose Algorithm~\ref{alg1} that computes a finitely generated set $X$ satisfying
\begin{align}\label{eq:sandwich}
    \NE(f,\mathbb X) \subseteq X  \subseteq \eNE(f,\mathbb X),
\end{align}
as proven in Theorem~\ref{sandwicheps}.
Thus, one obtains an even better approximation of the set of Nash equilibria than the set of all $\epsilon$-approximate Nash equilibria would provide. This will be illustrated in Section~\ref{sec:ex}, see in particular Figure~\ref{fig:figure1} and~\ref{fig:ex2alltogether}.

In the following, we will make the following assumption concerning the structure of the constraint set in addition to considering convex games (Assumption~\ref{ass:convex}). 

\begin{assumption}{\label{asspolycase}}
   The shared constraint set $\bbx \subseteq \prod_{j=1}^N \mathcal{X}_j$ is a nonempty polytope of the form $\bbx=\conv \{\bar{x}^1,...,\Bar{x}^k \}$ for some $k\in\N$. 
\end{assumption}

\begin{remark}
\begin{itemize}
\item
Assumption~\ref{asspolycase} can be equivalently formulated with linear inequalities, i.e., there is some $A \in \mathbb R^{p \times \sum_{j=1}^N n_j}$ and $b \in \mathbb R^p, p \in \mathbb N$ such that $\bbx=\{x \in \mathbb R^{\sum_{j=1}^N n_j}| Ax \leq b\}$.
\item Assumption~\ref{asspolycase} is satisfied for many games found in the literature (see, e.g., \cite[GNEP (21)]{nabetani2011parametrized} and \cite[Proposition 3]{braouezec2021economic}). Included are also special cases like mixed strategies (without further constraints) as then $\bbx=[0,1]^N$ or, more generally, box constraints.
\item In Remark~\ref{rem:inconstr}, we will drop Assumption~\ref{asspolycase} and consider instead the case where each player $i$ has an independent convex constraint set $\bbx_i \subseteq \mathcal{X}_i$.
\end{itemize}
\end{remark}

For the overall goal to algorithmically compute a set $X$ satisfying~\eqref{eq:sandwich}, one has, by equations~\eqref{eq:shared} and~\eqref{eq:shared_eps}, to find for each player $i$ a set $X_i$ satisfying
\begin{align}\label{eq:sandwich2}
   \argmin\{(x_{-i},-x_{-i},f_i(x)) \; | \; x \in \bbx\} \subseteq X_i \subseteq \epsilon\argmin\{(x_{-i},-x_{-i},f_i(x)) \; | \; x \in \bbx\}.
\end{align}
Then, the desired $X$ is obtained by setting $X=\bigcap_{i=1}^N X_i$.
The sets $X_i$ will be constructed in three steps. Recall that each player $i$ considers the multi-objective problem~\eqref{eq:vop}. Let us denote by 
\[
	\mathcal{P}_i:=\{(x_{-i},-x_{-i},f_i(x)) \; | \;x \in \mathbb X\}+\mathbb R^{m_i}_+
\]  
the upper image of problem~\eqref{eq:vop}. Note that by the assumptions of this section, the set $\mathcal{P}_i$ is closed, so the closure in the definition of the upper image is not needed here. In a first step one needs to compute a polyhedral inner approximation $\mathcal{P}^{In}_i$ of $\mathcal{P}_i$ such that
\begin{align}{\label{innerappr}}
    \mathcal{P}_i^{In} \subseteq \mathcal{P}_i \subseteq \mathcal{P}^{In}_i-\varepsilon \{c\}
\end{align}
for $c=\bar{c}_i:=(0,....,0,1)^\top \in\mathbb R^{m_i}_+$ and $\varepsilon>0$. As $\bar{c}_i\notin \int\R^{m_i}_+$, the algorithm in~\cite{LRU14} cannot be applied directly, but can be modified to our setting. The details of this modification are given in Subsection~\ref{subsec:mod}; this addresses problem~(iii) above. The second step is to sort out faces of this polyhedral approximation that are only weakly efficient in a certain sense; this addresses problem~(i) above. Details are given in Subsection~\ref{sec:maxefficientFaces}. The third step is then to approximate the set of preimages that lie below the remaining maximal efficient faces yielding a set of $\epsilon$-Pareto optimal points $X_i$ satisfying~\eqref{eq:sandwich2}. This addresses problem~(ii) above. Details will be given in Subsection~\ref{sec:Proj}.

\subsection{Computing a polyhedral approximation of the upper image}
\label{subsec:mod}

Algorithm~1 and~2 of~\cite{LRU14} allow for the computation of a polyhedral $\varepsilon$-approximation~\eqref{innerappr} to problem~\eqref{eq:vop} for a direction $c_i\in \int\R^{m_i}_+$. Since, by Theorem~\ref{thm:nash-approx}, we have to use $c_i=\bar{c}_i:=(0,....,0,1)^\top \notin \int\R^{m_i}_+$ instead, we cannot apply~\cite[Algorithm~1 or~2]{LRU14} directly. Let us consider Algorithm~1 of~\cite{LRU14} first. Algorithm~2 of~\cite{LRU14} is considered in Remark~\ref{rem:algo2} below. A careful inspection of~\cite[Algorithm 1]{LRU14} reveals that the assumption $c_i\in \int\R^{m_i}_+$ is only used in \cite[Proposition 4.4]{LRU14} to show that the so-called Pascoletti-Serafini scalarization of problem~\eqref{eq:vop} for $v\in\R^{m_i}$, that is, 
\begin{align*}{\tag{$P_2(v)$}}
    \min \{z \in \mathbb R\; | \; (x_{-i},-x_{-i},f_i(x))-z \bar{c}_i-v \leq 0,\; x \in \bbx \},
\end{align*} 
is feasible and strong duality is satisfied between $(P_2(v))$ and its Lagrange dual problem. This in turn is used to prove the correctness of~\cite[Algorithm~1]{LRU14} in~\cite[Theorem~4.9]{LRU14}. 
For the direction $\bar{c}_i=(0,....,0,1)^\top$, feasibility and strong duality might fail for the points $v$ considered in the course of the algorithm.
We will now introduce a modification to~\cite[Algorithm~1]{LRU14} that allows for the specific problem of interest, problem~\eqref{eq:vop}, to
restore feasibility and strong duality for $(P_2(v))$ and its dual. This modification takes advantage of the fact that the zero entries of $\bar{c}_i$ correspond to linear components in the objective of problem~\eqref{eq:vop}. 

Let us now explain the details of this modification. Algorithm 1 of~\cite{LRU14} starts with an initialization phase where an initial outer approximation $\mathcal{P}_0$ of $\mathcal{P}_i$ is computed via solving weighted sum scalarizations. The solutions of the corresponding weighted sum scalarizations are added to an initial set $\bar{\mathcal{X}}$. This is done in line~2 and~3 of \cite[Algorithm~1]{LRU14} and remains the same in the modified version. It is shown in \cite{LRU14} that the initial outer approximation $\mathcal{P}_0$ is pointed. 

However, before entering the iteration phase in line~4 of~\cite[Algorithm~1]{LRU14}, we will add the following two lines to the algorithm, where we denote by $a_i:=\frac{{m_i}-1}{2}$:
\begin{align}{\label{outerapprupdate}}
    \mathcal{P}_0=\mathcal{P}_0 \cap \{y \in \mathbb R^{m_i}\; | \; y_{1:a_i}=-y_{a_i+1:m_i-1} \} \cap \{y \in \mathbb R^{m_i}\; | \; y_{1:a_i} \in \conv \{\Bar{x}^1_{-i},...,\Bar{x}^k_{-i} \} \},
\end{align}
\begin{align}{\label{barXupdate}}
    \bar{\mathcal{X}}=\bar{\mathcal{X}} \cup \{\bar{x}^1,...,\Bar{x}^k \}.
\end{align}
Recall that $\bbx=\conv \{\bar{x}^1,...,\Bar{x}^k \}$ for $k\in\N$ (Assumption~\ref{asspolycase}). This completes the now modified initialization phase. 
The rest of the algorithm (i.e.\ the iteration phase of~\cite[Algorithm~1]{LRU14}) remains unchanged.

Let us now comment on these two modifications. Line~\eqref{outerapprupdate} makes additional cuts to the outer approximation $\mathcal{P}_0$.
The first cut with the linear subspace $\{y \in \mathbb R^{m_i}\; | \; y_{1:a_i}=-y_{a_i+1:m_i-1} \}$ is using the particular structure of problem~\eqref{eq:vop}, utilizing that the first $a_i$ components of the objective function are the negative of the next $a_i$ components. Thus, the upper image has to lie in this linear subspace and this fact can be used for the initial outer approximation $ \mathcal{P}_0$ already. The second cut with the polyhedron $\{y \in \mathbb R^{m_i}\; | \; y_{1:a_i} \in \conv \{\Bar{x}^1_{-i},...,\Bar{x}^k_{-i} \} \}$ is using that, by Assumption~\ref{asspolycase}, $\bbx=\conv \{\bar{x}^1,...,\Bar{x}^k \}$ for $k\in\N$ and the fact that the first $a_i$ components of the objective function of problem~\eqref{eq:vop} are $x_{-i}$.
Since both, the linear space $\{y \in \mathbb R^{m_i}\; | \; y_{1:a_i}=-y_{a_i+1:m_i-1} \}$ and the polyhedron $\{y \in \mathbb R^{m_i}\; | \; y_{1:a_i} \in \conv \{\Bar{x}^1_{-i},...,\Bar{x}^k_{-i} \} \}$, are supersets of the image set $\{(x_{-i},-x_{-i},f_i(x))\; | \;x \in \bbx \}$ of \eqref{eq:vop}, the updated set $\mathcal{P}_0$ remains still an outer approximation of this image set. (Note that one could add $\mathbb R^{m_i}_+$ to $\mathcal{P}_0$ obtained in~\eqref{outerapprupdate} in order to obtain an $\mathcal{P}_0$ that would still be an outer approximation of $\mathcal{P}_i$, but since it is enough for it to be an outer approximation of the image set, the addition of $\mathbb R^{m_i}_+$ is not needed for the algorithm to work correctly.) Further, the updated set $\mathcal{P}_0$ is still a pointed polyhedron.
Line~\eqref{barXupdate} updates the set $\bar{\mathcal{X}}$ by adding the vertices of the feasible set $\bbx$. This does not affect~\cite[Algorithm 1]{LRU14} or its correctness and only changes the set $\bar{\mathcal{X}}$ outputted by the algorithm. It is however crucial for step~3 of Algorithm~\ref{alg1} introduced below, see also Section~\ref{sec:Proj}, as it has an important consequence as detailed in Remark~\ref{rem1}.

In the second phase of~\cite[Algorithm~1]{LRU14} the outer approximation $\mathcal{P}_0$ is iteratively updated until after termination of the algorithm, \eqref{innerappr} is satisfied. We prove in the following that the iteration phase of \cite[Algorithm 1]{LRU14} works also correctly for problem~\eqref{eq:vop} with respect to the boundary direction $\bar{c}_i$, if the modification of the initialization phase, i.e.\ the one with the two additional lines~\eqref{outerapprupdate} and~\eqref{barXupdate}, is used.

The idea of the iteration phase of~\cite[Algorithm 1]{LRU14} is that in each iteration it is checked if the distance of the vertices of the current outer approximation $\mathcal{P}_0$ to the upper image $\mathcal{P}_i$ is less or equal than $\varepsilon$. If for each vertex the distance is less or equal than $\varepsilon$ then the algorithm stops. If a vertex of $\mathcal{P}_0$ is found to have a distance larger than $\varepsilon$, then the set $\mathcal{P}_0$ is updated. To check for each vertex $v \in \mathcal{P}_0$ the distance to $\mathcal{P}_i$, the Pascoletti-Serafini scalarization $(P_2(v))$ is considered. Since $v$ is always chosen among the vertices of $\mathcal{P}_0$, we can, using~\eqref{outerapprupdate}, rewrite $(P_2(v))$ as
\begin{align*}
    \min \{z \in \mathbb R\; | \; x_{-i}=v_{1:a_i}, f_i(x)-z-v_{m_i} \leq 0, Ax \leq b \}.
\end{align*}
Note that since $\bbx$ is a polyhedral set, the constraint $f_i(x)-z-v_{m_i} \leq 0$ is the only nonlinear constraint for $(P_2(v))$. The Lagrange dual problem to $(P_2(v))$ is given by
\begin{align*}{\tag{$D_2(v)$}}
    \max \{ \inf_{x \in \bbx} \{u^\top(Ax-b)+w^\top(x_{-i},-x_{-i},f_i(x))^\top \}-w^\top v\; | \; u \geq 0,\; w^\top \bar{c}_i=1, \; w \geq 0\}.
\end{align*}
To update the outer approximation $\mathcal{P}_0$, optimal solutions of both $(P_2(v))$ and $(D_2(v))$ are required. In the following lemma, the existence of optimal solutions for both problems and strong duality is proven. 

\begin{lemma}{\label{PSduality}}
Let Assumptions~\ref{ass:convex} and~\ref{asspolycase} be satisfied.
Let $\mathcal{P}_0$ be an outer approximation of the image set $\{(x_{-i},-x_{-i},f_i(x))\; | \;x \in \bbx \}$ of \eqref{eq:vop} satisfying~\eqref{outerapprupdate}. For every $v \in \mathcal{P}_0$ there exist optimal solutions $(x^v,z^v)$ and ($u^v,w^v)$ to $(P_2(v))$ and $(D_2(v))$ respectively, and the optimal values coincide.
\end{lemma}

\begin{proof}
Consider $v \in \mathcal{P}_0$.   By \eqref{outerapprupdate} we know that $v_{1:a_i} \in \conv \{\Bar{x}^1_{-i},...,\Bar{x}^k_{-i} \}$. Thus, there are $\lambda^1,...,\lambda^k \geq 0$ with $\sum_{j=1}^k \lambda^j=1$ such that $v_{1:a_i}=\sum_{j=1}^k \lambda^j \Bar{x}^j_{-i}$. Set $\bar{x}=\sum_{j=1}^k \lambda^j \Bar{x}^j\in\bbx$ and choose $\bar{z} \in \mathbb R$ such that   $f_i(\bar{x})-\bar{z}-v_{m_i}\leq 0$.
 Then $(\bar{x},\bar{z})$ is feasible for $(P_2(v))$. Since the feasible set of $(P_2(v))$ is compact (by the same arguments as in the proof of~\cite[Proposition~4.4]{LRU14}), an optimal solution exists. In order to guarantee strong duality we need to find a feasible element of $(P_2(v))$ that satisfies the weak Slater condition, i.e.\ for all nonaffine constraints the inequality constraint needs to be strict. However, since the only nonaffine constraint of $(P_2(v))$ is  $f_i({x})-{z}-v_{m_i}\leq 0$ we can always find $z^* \in \mathbb R$ such that $f_i(\bar{x})-z^*-v_{m_i}<0$. Then $(\bar{x},z^*)$ satisfies the weak Slater condition which implies strong duality.
\end{proof} 

This result guarantees (by replacing in the proof of~\cite[Theorem~4.9]{LRU14}, statement~\cite[Proposition~4.4]{LRU14} by Lemma~\ref{PSduality}) that the iteration phase works correctly under the changes done in the initialization phase. Thus, the modification of~\cite[Algorithm~1]{LRU14} proposed here outputs, at termination, a finite set $\bar{\mathcal{X}}$ such that $\mathcal{P}^{In}_i:=\conv \{(x_{-i},-x_{-i},f_i(x))\; | \;x \in \Bar{\mathcal{X}} \}+\mathbb R^{m_i}_+$ satisfies \eqref{innerappr}.

\begin{remark}{\label{rem1}}
Denote the elements of the finite set $\Bar{\mathcal{X}}$ produced at termination by the modification of~\cite[Algorithm~1]{LRU14} by $\{{x}^1,...,{x}^s\}$ for $s\in\N$ and consider $\mathcal{P}^{In}_i:=\conv \{(x_{-i},-x_{-i},f_i(x))\; | \;x \in \Bar{\mathcal{X}} \}+\mathbb R^{m_i}_+$. Since for any vertex $\bar{x}$  of $\bbx$, we have, due to \eqref{barXupdate}, $(\bar{x}_{-i},-\bar{x}_{-i},f_i(\bar{x})) \in \mathcal{P}^{In}_i$, we conclude that for any $x \in \bbx$ it holds $x_{-i} \in \conv \{{x}^1_{-i},...,{x}^s_{-i}\}$. As such, for any Pareto optimal element $x^*$ of \eqref{eq:vop} we can express $x^*_{-i}$ via convex combinations of ${x}^1_{-i},...,{x}^s_{-i}$. This is necessary to prove the correctness of Algorithm~\ref{alg1} proposed in Subsection~\ref{Algosection}.
\end{remark}

\begin{remark}{\label{rem:algo2}}
Let us now consider the dual algorithm, Algorithm~2 of~\cite{LRU14}. The assumption $c_i\in \int\R^{m_i}_+$ is only used in \cite[Proposition 3.10]{LRU14}, which is needed to show the correctness of this algorithm. One can prove that for the considered convex multi-objective optimization problem~\eqref{eq:vop}, \cite[Proposition 3.10]{LRU14} holds also for direction $c_i=\bar{c}_i:=(0,....,0,1)^\top \notin \int\R^{m_i}_+$. Hence, Algorithm~2 of~\cite{LRU14} can be applied directly to problem~\eqref{eq:vop}. However, the addition of line~\eqref{barXupdate} is also needed here as it is crucial for step~3 of Algorithm~\ref{alg1} introduced below.
\end{remark}

\subsection{Computing the maximal efficient faces}
\label{sec:maxefficientFaces}
In Subsection~\ref{subsec:mod} we obtained a polyhedral $\varepsilon$-approximation $\mathcal{P}^{In}_i$ of problem~\eqref{eq:vop} satisfying~\eqref{innerappr} for the direction $\bar{c}_i$. Note that the boundary of $\mathcal{P}^{In}_i$ is the set of all weakly minimal elements of $\mathcal{P}^{In}_i$. The next step is to sort out all weakly minimal elements of $\mathcal{P}^{In}_i$ that are not minimal. This addresses problem~(i) above. To do so, we apply the concept of maximal efficient faces in vector optimization. For some general multi-objective problem of the form~\eqref{eq:VOPg}, a face $F$ of $\bbx$ is called efficient if it contains Pareto optimal elements only. An efficient face $F$ of $\bbx$ is called maximal efficient if there is no face $G$ of $\bbx$ such that $F \subsetneq G$. Now consider the following linear multi-objective optimization problem
\begin{align}
\label{faces}
    \min\{y \; | \; y \in \mathcal{P}_i^{In}\}
\end{align}
with ordering cone $\mathbb R^{m_i}_{+}$. One can easily see that the union of all maximal efficient faces of ~\eqref{faces} equals the set of minimal elements of $\mathcal{P}^{In}_i$ since $\mathcal{P}^{In}_i$ is the upper image of ~\eqref{faces}. As this is now a linear problem due to the polyhedral structure of $\mathcal{P}^{In}_i$, one can use existing methods, see e.g.~\cite{Armand93,VanTu17,tohidi2018adjacency}, to compute all maximal efficient faces $F^1,...,F^{k_i}$ of $ \mathcal{P}_i^{In}$, where $k_i\in\N$. 
\begin{remark}{\label{rem:maxfacesbd}}
Note that any maximal efficient face $F$ for problem ~\eqref{faces} is a polytope: Since $F$ is a face of the polyhedral feasible set $\mathcal{P}^{In}_i$ it is also a polyhedron. Assume now that $F$ is not bounded, i.e.\ $F=\conv V+\cone D$ where $V \subseteq \mathcal{P}^{In}_i$ and $D \subseteq \mathbb R^{m_i}_+ \setminus \{0\}$ for finite sets $V$ and $D$. Then there is some $y \in F$ with $y=\Bar{y}+d$ where $\Bar{y} \in \conv V$ and $d \in \mathbb R^{m_i}_+ \setminus \{0\}$. Then it must be that $\Bar{y} \leq y$ and $\Bar{y} \neq y$ which is a contradiction to $y$ being Pareto optimal for problem ~\eqref{faces}.
\end{remark}

\subsection{Approximating the set of Pareto optimal points}
\label{sec:Proj}
The next step is to approximate the set of Pareto optimal points of problem~\eqref{eq:vop}. 
In detail, we want to approximate the set of preimages whose images lie below the maximal efficient faces $F^j$, $j=1,...,k_i$, computed in Section~\ref{sec:maxefficientFaces}, yielding a set of $\epsilon$-Pareto optimal points $X_i$ satisfying~\eqref{eq:sandwich2}.
The aim is to compute
\begin{align}
\label{cp}
	\bar{X}^j_i:=\{x \in \mathbb{X}\; | \; \exists y \in F^j: (x_{-i},-x_{-i},f_i(x)) \leq_{\mathbb R^{m_i}_{+}} y\}
\end{align}
for each maximal efficient face $F^j$, where $j=1,...,k_i$. Note that this is a bounded convex projection problem as considered in~\cite{kr22,kr23} as it is of the form: 
$
\text{compute } \{ x \in \mathbb{X} \; | \;  \exists y \in F^j : (x,y) \in S \}
$
for $S=\{(x,y)\in \mathbb{X}\times  F^j \; | \; (x_{-i},-x_{-i},f_i(x)) \leq_{\mathbb R^{m_i}_{+}} y\}$. The set $S$ is bounded due to Assumption~\ref{asspolycase} and Remark~\ref{rem:maxfacesbd}. As the set $S$ is convex (and not polyhedral in general), the sets $\bar{X}^j_i$ cannot be computed exactly, but can only be approximated by the methods in e.g.~\cite{SZC18,LZS21,kr22} or~\cite[Algorithm~5.1]{kr23}. Note that~\cite{SZC18,LZS21,kr22} solve an associated convex multi-objective problem in dimension $d+1$ for $d=\sum_{j=1}^N n_j$, whereas~\cite[Algorithm~5.1]{kr23} works in dimension $d$.
For a fixed error level $\varepsilon_2>0$ one obtains a finite set $\hat{X}^j_i  \subseteq\mathbb{X}$ satisfying 
\begin{align}
\label{innerproj}
    \conv \hat{X}^j_i \subseteq\bar{X}^j_i\subseteq \conv \hat{X}^j_i+B_{\varepsilon_2},
    \end{align}
see e.g.~\cite[Theorem~3.11]{kr22} and~\cite[Theorem~5.4]{kr23}, where $B_{\varepsilon_2}$ is a closed $\varepsilon_2$-ball around the origin (in the $L_1$ norm). In detail,~\cite[Algorithm~5.1]{kr23} outputs a finite  $\varepsilon_2$-solution $\bar S\subseteq S$ of the bounded convex projection problem~\eqref{cp} such that 
$\hat{X}^j_i:=\{x \in \mathbb{X} \; | \; (x,y)\in  \bar S\}$ (i.e. the collection of the $x$ components of the finitely many vectors $(x,y)\in  \bar S$)  satisfies~\eqref{innerproj}.

We will show in Lemma~\ref{sandwich} below that the union of the sets $\bar{X}^j_i$ from the convex projection satisfy~\eqref{eq:sandwich2} with respect to the error level $ \varepsilon_1$. As the convex projections $\bar{X}^j_i$ cannot be computed exactly, we will then show that the union of the approximate sets, i.e., 
\begin{align*}
	X_i:=\bigcup_{j=1}^{k_i} ( \conv \hat{X}^j_i+B_{\varepsilon_2})\cap \bbx
\end{align*}
is the desired subset of $\epsilon$-Pareto optimal points of problem~\eqref{eq:vop} satisfying~\eqref{eq:sandwich2} for an adapted error level $\epsilon$, that involves the two error levels $\varepsilon_1, \varepsilon_2>0$ and the Lipschitz constant of the function $f_i$. Hence the set 
$
	X=\bigcap_{i=1}^N X_i
$ 
 will be the desired set yielding \eqref{eq:sandwich}:
 \begin{align*}
    \NE(f,\mathbb X) \subseteq X  \subseteq \eNE(f,\mathbb X).
\end{align*}

\subsection{Algorithm and main result}{\label{Algosection}}
We are now ready to combine the three steps described in the last subsections, present our algorithm, and prove the main result of this paper.

\begin{alg}{\label{alg1}}\begin{description}

\item[Input]
convex cost functions $f_1,...,f_N$ and shared polyhedral constraint set $\bbx$ satisfying Assumptions~\ref{ass:convex} and~\ref{asspolycase}, approximation levels $\varepsilon_1, \varepsilon_2>0$

\item[For each] $i=1,...,N$ do
\item[step 1:] Consider the convex multi-objective optimization problem \eqref{eq:vop}:
\begin{equation*}
\min\{(x_{-i},-x_{-i},f_i(x)) \; | \; x \in \bbx\}.
\end{equation*}
Compute an inner approximation $\mathcal{P}^{In}_i$ such that \eqref{innerappr} is satisfied with approximation level $\varepsilon_1>0$ using~\cite[Algorithm~1]{LRU14} with the modified initialization using~\eqref{outerapprupdate} and~\eqref{barXupdate} as described in Section~\ref{subsec:mod} or using~\cite[Algorithm~2]{LRU14} with modification~\eqref{barXupdate} as discussed in Remark~\ref{rem:algo2}.
\item[step 2:] Consider the linear multi-objective optimization problem~\eqref{faces}:
\begin{align*}
    \min\{y \; | \; y \in \mathcal{P}_i^{In}\}
\end{align*}
with ordering cone $\mathbb R^{m_i}_{+}$ and compute the maximal efficient faces $F^1,...,F^{k_i}$ (e.g.\ with algorithm from~\cite{Armand93,VanTu17,tohidi2018adjacency},  see also Section~\ref{sec:maxefficientFaces}). 
\item[step 3:]
For each $j=1,...,k_i$ and given approximation level $\varepsilon_2>0$ compute a finite  $\varepsilon_2$-solution $\bar S$ of the convex projection problem~\eqref{cp}:
\begin{align*}
	\text{compute } \bar{X}^j_i:=\{x \in \mathbb{X}\; | \; \exists y \in F^j: (x_{-i},-x_{-i},f_i(x)) \leq_{\mathbb R^{m_i}_{+}} y\}.
\end{align*}
This can be done by e.g.~\cite{SZC18,LZS21,kr22} or~\cite[Algorithm~5.1]{kr23}, see also Section~\ref{sec:Proj}.
Then, $\hat{X}^j_i:=\{x \in \mathbb{X} \; | \; (x,y)\in  \bar S\}$ yields an $\varepsilon_2$-approximation of the set $\bar{X}^j_i$ in the sense of~\eqref{innerproj}.
Set 
    \begin{align*}
X_i=\bigcup_{j=1}^{k_i} ( \conv \hat{X}^j_i+B_{\varepsilon_2})\cap \bbx.
    \end{align*}
\item[Output:] $X=\bigcap_{i=1}^N X_i$  
\end{description}
\end{alg}

We can now state the main theorem of this paper. Recall that $L>0$ is the largest of the Lipschitz constants of the functions $f_i$.

\begin{theorem}{\label{sandwicheps}}
Let Assumptions~\ref{ass:convex} and~\ref{asspolycase} be satisfied.  Let $X \subseteq \mathbb R^{\sum_{i = 1}^N n_i}$ be the set computed by Algorithm~\ref{alg1} and let $\epsilon=\varepsilon_1+2L \varepsilon_2$. Then it holds 
\begin{align}\label{eq:sandwichth}
    \NE(f,\mathbb X) \subseteq X   \subseteq \epsilon\NE(f,\mathbb X).
\end{align}
\end{theorem}
In order to prove Theorem~\ref{sandwicheps} we need the following lemmata. The first one shows that the union of the exact sets $\bar{X}^j_i$ of the convex projection~\eqref{cp} satisfy~\eqref{eq:sandwich2} with respect to the error level $ \varepsilon_1$.

 \begin{lemma}{\label{sandwich}}
 Let Assumptions~\ref{ass:convex} and~\ref{asspolycase} be satisfied.
 For any $i=1,...,N$ it holds
    \begin{align*}
        \argmin \{(x_{-i},-x_{-i},f_i(x))\; | \;x \in \mathbb{X}\} \subseteq \bigcup_{j=1}^{k_i} \bar{X}^j_i \subseteq  \varepsilon_1\argmin \{(x_{-i},-x_{-i},f_i(x))\; | \;x \in \mathbb{X}\}.
    \end{align*}
\end{lemma}

\begin{proof}
    For the left inclusion consider $x^* \in \argmin \{(x_{-i},-x_{-i},f_i(x))\; | \;x \in \mathbb{X}\}$.\\ 
    Let $\mathcal{P}^{In}_i=\conv\{(x_{-i},-x_{-i},f_i(x))\; | \;x \in \Bar{\mathcal{X}} \}+\mathbb R^{m_i}_+$ for $ \Bar{\mathcal{X}}=\{{x}^1,...,{x}^s\}\subseteq \bbx$ with $s\in\N$ 
 be the inner approximation of $\mathcal{P}_i$ computed as described in Section~\ref{subsec:mod}. Due to Remark \ref{rem1}, we know there are $\lambda^1,...,\lambda^s \geq 0$ with $\sum_{l=1}^s \lambda^l=1$ such that $x^*_{-i}=\sum_{l=1}^s \lambda^l {x}^l_{-i}$. Thus, $(x^*_{-i},-x^*_{-i},\sum_{l=1}^s \lambda^l f_i({x}^l)) \in \mathcal{P}^{In}_i$. Then there is $\hat{y}$ such that $(x^*_{-i},-x^*_{-i},\hat{y})$ is minimal in $\mathcal{P}^{In}_i$: Assume there would be no such $\hat{y}$. Then for all $\alpha>0$ it would be $(x^*_{-i},-x^*_{-i},\sum_{l=1}^s \lambda^l f_i({x}^l))-\alpha \bar{c}_i \in \mathcal{P}^{In}_i$. In that case $\mathcal{P}^{In}_i$ would contain the line $\{(x^*_{-i},-x^*_{-i},\sum_{l=1}^s \lambda^l f_i({x}^l))+\alpha \bar{c}_i\; | \; \alpha \in \mathbb R\}$ which is a contradiction to $\mathcal{P}^{In}_i$ being pointed. Thus, there exists $\hat{y}$ such that $(x^*_{-i},-x^*_{-i},\hat{y})$ is minimal in $\mathcal{P}^{In}_i$ and therefore exists a maximal efficient face $F^j$ of $\mathcal{P}^{In}_i$ with $(x^*_{-i},-x^*_{-i},\hat{y}) \in F^j$. So there are $\hat{\lambda}^1,...,\hat{\lambda}^s \geq 0$ with $\sum_{l=1}^s \hat{\lambda}^l=1$ such that  $(x^*_{-i},-x^*_{-i},\hat{y})=(\sum_{l=1}^s \hat{\lambda}^l {x}^l_{-i},-\sum_{l=1}^s \hat{\lambda}^l {x}^l_{-i},\sum_{l=1}^s \hat{\lambda}^l f_i({x}^l))$. 
 Since 
 \begin{align}{\label{x}}
f_i(x^*) \leq f_i(\sum_{l=1}^s \hat{\lambda}^l {x}^l) \leq \sum_{l=1}^s \hat{\lambda}^l f_i({x}^l)=\hat{y},
\end{align}
by $x^* \in \argmin \{(x_{-i},-x_{-i},f_i(x))\; | \;x \in \mathbb{X}\}$, $\sum_{l=1}^s \hat{\lambda}^l {x}^l\in\bbx$, and $f_i$ convex, it is $x^* \in \bar{X}^j_i$.

 For the second statement let $x^* \in \bigcup_{j=1}^{k_i} \bar{X}^j_i$. Then $x^* \in \bar{X}^j_i$ for some $j$. It is $(x^*_{-i},-x^*_{-i},f_i(x^*)) \leq y$ for $y \in F^j$. Note that $y$ is a minimal element in $\mathcal{P}^{In}_i$. So, $(x^*_{-i},-x^*_{-i},f_i(x^*)-\varepsilon_1) \leq y-(0,0,\varepsilon_1)$ and $y-(0,0,\varepsilon_1)$ is a minimal element in $\mathcal{P}^{Out}_i$. This means either $(x^*_{-i},-x^*_{-i},f_i(x^*)-\varepsilon_1) \notin \mathcal{P}_i$ or $(x^*_{-i},-x^*_{-i},f_i(x^*)-\varepsilon_1)$ is a minimal element in $ \mathcal{P}_i$. 
 In any case there is no $x \in \mathbb X$ with $x_{-i}=x^*_{-i}$ and $f_i(x)+\varepsilon_1<f_i(x^*)$.   
 \end{proof}

As the convex projections $\bar{X}^j_i$ cannot be computed exactly, we will now prove the corresponding result for the approximate sets $X_i=\bigcup_{j=1}^{k_i} ( \conv \hat{X}^j_i+B_{\varepsilon_2})\cap \bbx$, where the error level has to be adjusted by the error of the approximate convex projection $\varepsilon_2$. As this error is made in the preimage space, it needs to be translated into the image space of the convex multi-objective optimization problem~\eqref{eq:vop}. For this step the Lipschitz continuity of $f_i$ implied by Assumption~\ref{ass:convex} is necessary.
\begin{lemma}{\label{epsbound}}
Let Assumptions~\ref{ass:convex} and~\ref{asspolycase} be satisfied. Let $\epsilon=\varepsilon_1+2L \varepsilon_2$.
 For any $i=1,...,N$ it holds
    \begin{align*}
        \argmin \{(x_{-i},-x_{-i},f_i(x))\; | \;x \in \mathbb{X}\} \subseteq X_i \subseteq  \epsilon\argmin \{(x_{-i},-x_{-i},f_i(x))\; | \;x \in \mathbb{X}\}.
    \end{align*}
\end{lemma}
\begin{proof}
    Fix $i\in\{1,...,N\}$ and let $\Bar{x} \in X_i$. Then, there is some $j\in\{1,...,k_i\}$ such that $\Bar{x} \in \conv \hat{X}^j_i+B_{\varepsilon_2}$.
    So, $\bar{x}=x^*+b^*$ for $x^* \in \conv \hat{X}^j_i$ and $\norm{b^*} \leq \varepsilon_2$. Since $\conv \hat{X}^j_i$ is an inner approximation of $\bar{X}^j_i$, see~\eqref{innerproj}, it follows by Lemma~\ref{sandwich} that $x^*$ is $\varepsilon_1$-Pareto optimal for~\eqref{eq:vop}. Thus, $f_i(x^*) \leq f_i(x)+\varepsilon_1$ for any $x \in \bbx$ with $x_{-i}=x^*_{-i}$ and $f_i(x) \leq f_i(x^*)$. Now let $x \in \mathbb X$ with $x_{-i}=\bar{x}_{-i}=x^*_{-i}+b^*_{-i}$ and $f_i(x) \leq f_i(\bar{x})$. It holds 
    \begin{align*}
        f_i(x)&+\varepsilon_1+2L\varepsilon_2-f_i(\bar{x})=f_i(x_i,x^*_{-i}+b^*_{-i})+\varepsilon_1+2L \varepsilon_2-f_i(x^*+b^*)\\
        &=f_i(x_i,x^*_{-i}+b^*_{-i})-f_i(x_i,x^*_{-i})+L  \varepsilon_2+f_i(x^*)-f_i(x^*+b^*) + L \varepsilon_2+f_i(x_i,x^*_{-i})-f_i(x^*)+\varepsilon_1 \geq 0,
    \end{align*}
    where we used that $|f_i(x_i,x^*_{-i}+b^*_{-i})-f_i(x_i,x^*_{-i})| \leq L  \varepsilon_2$ and $|f_i(x^*)-f_i(x^*+b^*)| \leq  L  \varepsilon_2$ due to Lipschitz continuity of $f_i$ as well as $f_i(x_i,x^*_{-i})+\varepsilon_1 \geq f_i(x^*)$ due to $x^*$ being $\varepsilon_1$-Pareto optimal.
    Thus, $f_i(x)+(\varepsilon_1 +2L\varepsilon_2) \geq f_i(\bar{x})$ which shows that $\Bar{x} \in (\varepsilon_1 +2L\varepsilon_2)\argmin \{(x_{-i},-x_{-i},f_i(x))\; | \;x \in \bbx\}$. 
\end{proof} 

We are now ready to prove Theorem~\ref{sandwicheps}.

\begin{proof}[Proof of Theorem~\ref{sandwicheps}]
Let us first prove the left hand side inclusion. Let $x^* \in \NE(f,\bbx)$. By Theorem~\ref{thm:nash}, it follows that $x^* \in \bigcap_{i=1}^N \argmin \{(x_{-i},-x_{-i},f_i(x))\; | \; x \in \bbx \}$. Then, for each $i=1,...,N$ there is by Lemma~\ref{sandwich} a $j=j(i)$ such that $x^* \in \bar{X}^j_i$. This implies, by~\eqref{innerproj}, $x^* \in \conv \hat{X}^j_i+B_{\varepsilon_2}$. As 
$x^* \in \NE(f,\bbx)$ also implies $x^* \in \bbx$, this proves the claim by the definition of the sets  $X_i$ and $X$.
Let us now prove the right hand side inclusion. Let $x^* \in X$, thus $x^* \in X_i$ for all $i=1,...,N$.
By Lemma~\ref{epsbound} and Theorem~\ref{thm:nash-approx} the claim follows.
\end{proof}

Let us conclude the section with several remarks.

\begin{remark}\label{rem:dim_red}
Recall that Theorem~\ref{thm:nash} was given in~\cite{FR23} in a more general framework and in a different formulation using a non-standard ordering cone (see problem~(5) with ordering cone~(7) in~\cite[Cor.~2.8]{FR23}), which can be equivalently written as
\begin{equation}\label{eq:voplowdim}
\mathbb R_+ \Bar{c}_i-\min\{(x_{-i},f_i(x)) \; | \; x \in \bbx\},
\end{equation}
where one is minimizing just in direction $\Bar{c}_i=(0,...,0,1)^\top \in \mathbb R^{\sum_{j=1, j \neq i}^N n_j+1}$. That is, the ordering cone of problem~\eqref{eq:voplowdim} is not the natural ordering cone, but $\mathbb R_+ \Bar{c}_i=\{r\Bar{c}_i\; | \; r\in \mathbb R_+ \}$.

Thus, instead of problem~\eqref{eq:vop} one could also consider problem~\eqref{eq:voplowdim}, which is lower dimensional (one dimension lower than the dimension reduction obtained already by \eqref{eq:vop2}), but has a non-standard ordering cone. One can, however, not expect a significant run time reduction of the proposed algorithm for this lower dimensional problem as the scalar problems solved within the algorithms in step~1 and~3 of the Algorithm~\ref{alg1} are basically unaffected by this change. Step~2 would be in lower dimensions, but the run time bottle neck of Algorithm~\ref{alg1} are rather steps~1 and~3.
\end{remark}

\begin{remark} \label{rem:inconstr}
Theorem~\ref{sandwicheps} also holds if we replace Assumption~\ref{asspolycase} by the following condition, keeping Assumption~\ref{ass:convex} unchanged:
Assume each player $i$ has an independent constraint set $\bbx_i=\{x_i \in \mathcal{X}_i\; | \; g_i(x_i) \leq 0\}$, where $g_i: \mathcal{X}_i \to \mathbb R$ is a convex function and $\bbx_i$ has nonempty interior. Thus, the shared constraint set is $\bbx=\bbx_1 \times ... \times \bbx_N=\{x \in \prod_{j = 1}^N \xcal_j\; | \; g_1(x_1) \leq 0,...,g_N(x_N) \leq 0\}$.
Constraints of this type are also frequently considered in the literature: they are called `orthogonal constraint sets' in~\cite{rosen65} and `classical games' in~\cite{braouezec2021economic}.
When replacing the polyhedral constraint Assumption~\ref{asspolycase} by the independent constraint assumption $\bbx=\bbx_1 \times ... \times \bbx_N$, one needs to adapt step~$1$ of Algorithm~\ref{alg1} as follows
\begin{description}
\item[step 1:] Compute a polyhedral outer approximation $P^{Out}_{-i}$ of $\bbx_{-i}:=\bbx_1 \times ... \times \bbx_{i-1} \times \bbx_{i+1} \times ... \times \bbx_N$. Denote the vertices of $P^{Out}_{-i}$ by $p^1,...,p^{l_i}$. Consider the convex multi-objective optimization problem
\begin{align*}
    \min \{(x_{-i},-x_{-i},f_i(x))\; | \;x_i \in \bbx_i,\; x_{-i} \in P^{Out}_{-i} \}.
\end{align*}
Compute an inner approximation $\mathcal{P}^{In}_i$ of its upper image such that \eqref{innerappr} is satisfied with approximation level $\varepsilon_1>0$ using~\cite[Algorithm~1]{LRU14} with the modified initialization, similar to Section~\ref{subsec:mod}, but using
\begin{align*}
    \mathcal{P}_0&=\mathcal{P}_0 \cap \{y \in \mathbb R^{m_i}\; | \; y_{1:a_i}=-y_{a_i+1:m_i-1} \} \cap \{y \in \mathbb R^{m_i}\; | \; y_{1:a_i} \in \conv \{p^1_{-i},...,p^{l_i}_{-i} \} \}
    \\
    \Bar{\mathcal{X}}&=\Bar{\mathcal{X}} \cup \{(x_i,p^j)\; | \;j=1,...,{l_i}\},
\end{align*}
where $x_i \in \bbx_i$ is chosen arbitrarily.
\end{description}
The rest of Algorithm~\ref{alg1} remains unchanged. The set $X$ outputted by this modified Algorithm~\ref{alg1} satisfies
\begin{align*}
    \NE(f,\mathbb X) \subseteq X   \subseteq \epsilon\NE(f,\mathbb X)
\end{align*}
for $\epsilon=\varepsilon_1+2L \varepsilon_2$ for a game with an independent constraint set as above and satisfying Assumption~\ref{ass:convex}.
The proofs are in analogy to the proof of Theorem~\ref{sandwicheps}. The key is that the independent constraint assumption ($\bbx_{i}\times \bbx_{-i}=\bbx$) ensures that the first inequality in~\eqref{x} still holds by $x^* \in \argmin \{(x_{-i},-x_{-i},f_i(x))\; | \;x \in \mathbb{X}\}$ and since $\sum_{l=1}^s \hat{\lambda}^l {x}^l= (\sum_{l=1}^s \hat{\lambda}^l {x}^l_{i}, x^*_{-i})\in \bbx_{i}\times \bbx_{-i}=\bbx$ is feasible, as ${x}^l_i\in\bbx_i$ ($l=1,...,s$) even though ${x}^l\notin\bbx$ is possible due to the use of the outer approximation $P^{Out}_{-i}$ for $\bbx_{-i}$.

The approximation error of the outer approximation $P^{Out}_{-i}$ of $\bbx_{-i}$ does not enter as, in step~$3$ of the algorithm, the intersection with the constraint set $\bbx$ is taken.
\end{remark}

\begin{remark} 
Note however, that for a more general constraint set $\bbx$ than that considered in Assumption~\ref{asspolycase} or Remark~\ref{rem:inconstr}, one might have that $\bbx$ is a strict subset of $ \bbx_{i}\times \bbx_{-i}$, i.e. $\bbx\subsetneq \bbx_{i}\times \bbx_{-i}$, where we define  $\bbx_{-i}:=\{ x_{-i} \in \prod_{j \neq i} \mathcal{X}_j\; | \; \exists \; x_i \in \mathcal{X}_i: (x_i,x_{-i}) \in \bbx \}$ and $\bbx_i:=\{x_i \in \mathcal{X}_i\; | \;\exists \; x_{-i} \in \prod_{j \neq i} \mathcal{X}_j: (x_i,x_{-i}) \in \bbx\}$. Therefore the first inequality in~\eqref{x} might fail as $\sum_{l=1}^s \hat{\lambda}^l {x}^l\notin \bbx$ is possible.
\end{remark}

\section{Examples}\label{sec:ex}
We will start in Sections~\ref{sec:ill} and~\ref{sec:search} with simple illustrative two player examples where the set of Nash equilibria, as well as the set of $\epsilon$-approximate Nash equilibria, can be computed explicitly. As such we directly investigate the sandwich principle \eqref{eq:sandwichth}. In Section~\ref{sec:pollution}, we will consider more complex examples from environmental economics involving two or three players. There, the set of Nash equilibria or $\epsilon$-approximate Nash equilibria is in general not known and Algorithm~\ref{alg1}  is used to compute a set $X$ that by Theorem~\ref{sandwicheps} is known to lie between the set of true Nash equilibria and the set of $\epsilon$-approximate Nash equilibria.

\subsection{Illustrative examples}\label{sec:ill}

\begin{example}
\label{ex:1}
Consider a two player game, where each player chooses a real-valued strategy, thus $\mathcal{X}_1=\mathcal{X}_2=\mathbb R$. Let the shared constraint set be $\bbx=\{(x_1,x_2) \in \mathbb R^2\; | \; x_1,x_2 \in [-1,1] \}$. Clearly, the set $\bbx$ is a compact polyhedral set. Consider for player $1$ the convex cost function $f_1(x_1,x_2)=\frac{1}{2}x^2_1-x_1(x_2+\frac{1}{2})+x^2_2$ and for player $2$ the convex cost $f_2(x_1,x_2)=\frac{1}{2}x^2_2+x_1 x_2+x^2_1$. 
One can show that $f_1$ and $f_2$ are Lipschitz continuous with constant $L=3$ on $\bbx$. Thus, Assumptions~\ref{ass:convex} and~\ref{asspolycase} are satisfied.
Due to the simple structure and low dimension of this example it is possible to find the set of Nash equilibria easily by hand. It consists of the unique Nash equilibrium $(\frac{1}{4},-\frac{1}{4})$.  However, we will apply Algorithm~\ref{alg1} to illustrate the sandwich result~\eqref{eq:sandwichth} of Theorem~\ref{sandwicheps} and thus find an approximation of the unique Nash equilibrium. We choose error levels $\varepsilon_1=0.01, \varepsilon_2=0.001$. Hence, the overall error is $\epsilon=\varepsilon_1+2L \varepsilon_2=0.016$. Due to the simple nature of the problem one can compute the exact set of $\epsilon$-approximate Nash equilibria as
\begin{align*}
    \epsilon\NE(f,\bbx)=&\{(x_1,x_2) \in \mathbb R^2| x_1 \in [\frac{1}{4}-\sqrt{0.032},\frac{1}{4}], x_2 \in [-x_1-\sqrt{0.032},x_1-\frac{1}{2}+\sqrt{0.032}] \}\\
    &\cup \{(x_1,x_2) \in \mathbb R^2| x_1 \in [\frac{1}{4},\frac{1}{4}+\sqrt{0.032}], x_2 \in [x_1-\frac{1}{2}-\sqrt{0.032},-x_1+\sqrt{0.032}] \}.
\end{align*}
Algorithm~\ref{alg1} computes a set $X\subseteq \mathbb R^2$ satisfying $\NE(f,\bbx) \subseteq X \subseteq \epsilon\NE(f,\bbx)$ for $\epsilon=0.016$. The three sets are depicted in Figure~\ref{fig:figure1}. Furthermore,
Figure~\ref{fig:figure1} shows the sets $X_i$ of both players $i=1,2$ satisfying~\eqref{eq:sandwich2} for $\varepsilon_1$. Their intersection $X_1 \cap X_2$ is the desired set $X$ which fulfills~\eqref{eq:sandwichth} for $\epsilon=0.016$.

\begin{figure}[h]
		\centering
		\subfloat[][]{\includegraphics[width=0.4\linewidth]{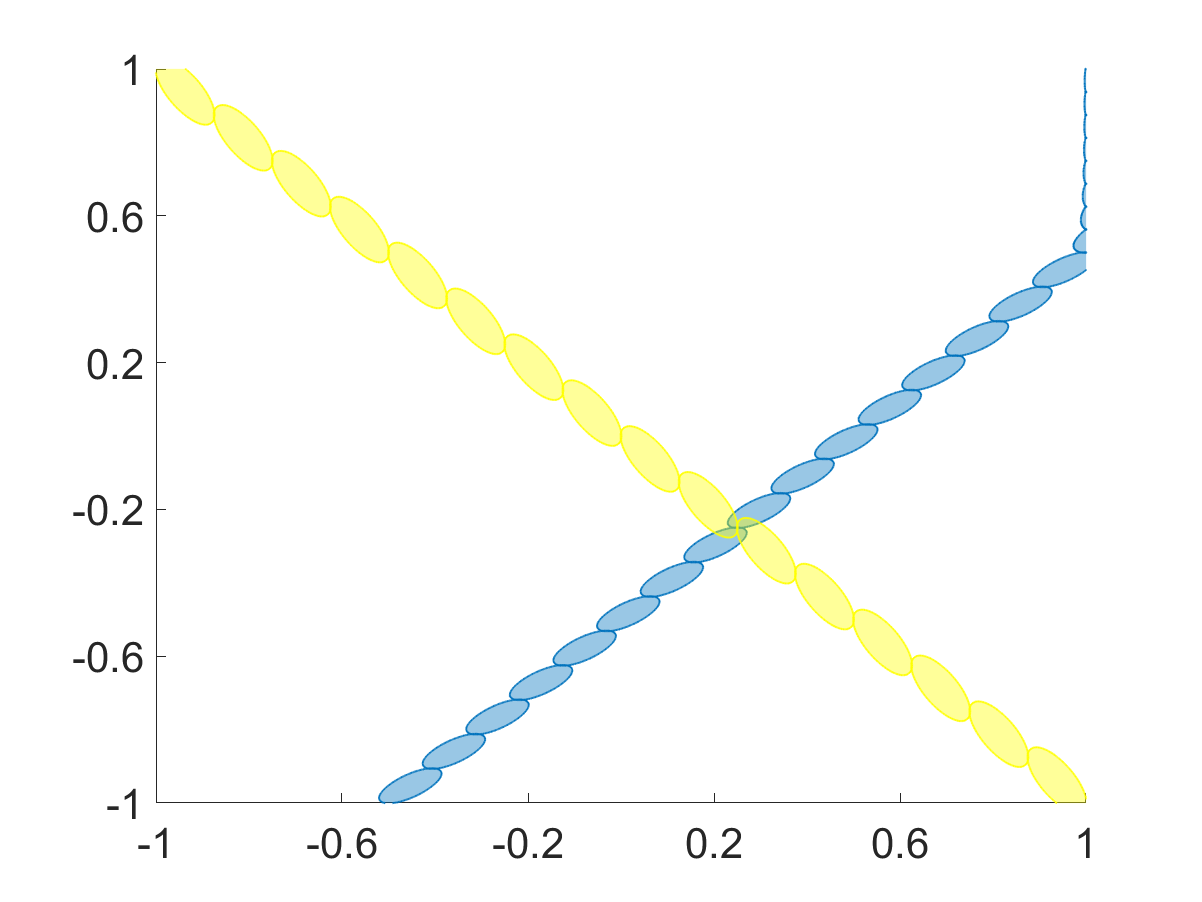}}
		\qquad
		\subfloat[][]{\includegraphics[width=0.4\linewidth]{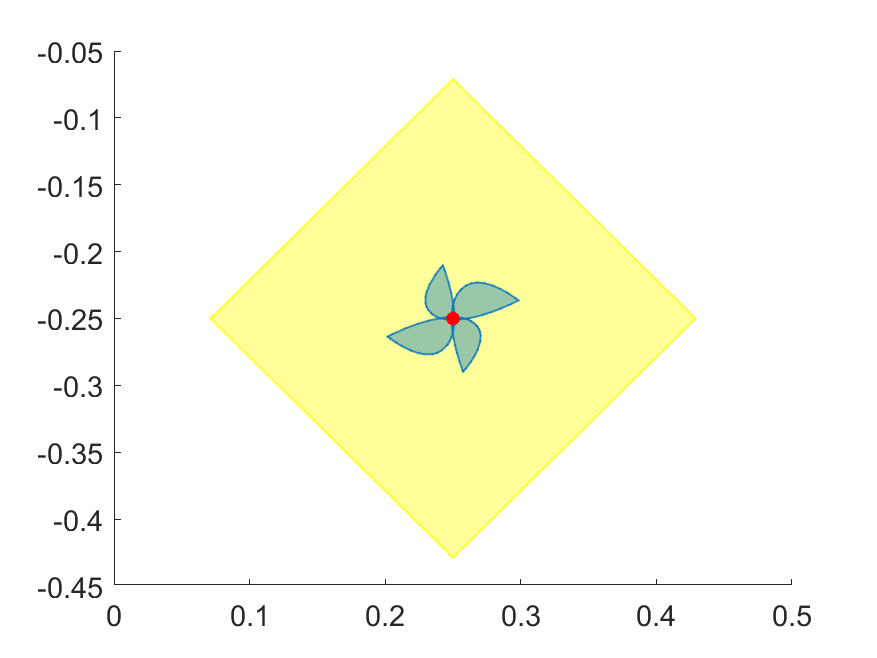}}
		\qquad
		\caption{(a): Computed sets $X_1$ for player $1$ (blue) and $X_2$ for player $2$ (yellow), (b): The \emph{singleton} set $\NE(f,\bbx)$ (red), the set $X$ computed by Algorithm~\ref{alg1} (blue) and the set $\epsilon\NE(f,\bbx)$ (yellow) for Example~\ref{ex:1}.} {\label{fig:figure1}}
	\end{figure}

\end{example}

\begin{example}
\label{ex:2}
Consider now the example provided within \cite[Section 2]{feinstein2022}. This is a parametric game which permits unique, multiple, and infinite Nash equilibria depending on the parameter chosen. 

Specifically, consider a two player game in which each player has real-valued strategies in the compact set $[0,1]$. Thus, the shared constraint set is a box-constraint $\bbx=\{(x_1,x_2) \in \mathbb R^2\; | \; x_1,x_2 \in [0,1] \}$. In order to introduce the convex cost functions, consider the compact parameter space $\mathcal{Y}:=[0,2]$. The cost functions for player $1$ and player $2$ are given by $f_1(x_1,x_2)=x_1(x_1-2yx_2)+x^2_2$ and $f_2(x_1,x_2)=(x_1-x_2)^2$ respectively, where $y \in \mathcal{Y}$. Note that the Lipschitz constant for this problem depends on the parameter $y$ is given by $L_y=\max \{4,2+2(y^2+10^{-4})\}$. For each $y \in [0,2]$, a Nash equilibrium exists and the set of Nash equilibria depending on $y$ is of the form
\begin{align*}
  \NE(f,\bbx)=  \begin{cases}
  \{(0,0)\}, \hspace{0.2cm} &y<1\\
  \{x \in [0,1]^2\; | \;x_1=x_2\}, \hspace{0.2cm} &y=1\\
  \{(0,0), (1,1)\}, \hspace{0.2cm} &y>1.
  \end{cases}
\end{align*}
For $y \in [0,1)$ we obtain a unique Nash equilibrium. If $y=1$ the set of Nash equilibria is a line. For $y \in (1,2]$ the set of Nash equilibria consists of two isolated points.

In order to demonstrate that our methods work independently of the set of Nash equilibria being a singleton, a finite set or an infinite set,
we apply Algorithm~\ref{alg1} for different values of $y$ given by $y \in \{0.5,1,1.5\}$. The chosen error levels are $\varepsilon_1=\varepsilon_2=0.001$. Figure~\ref{figure2} shows, for both players, the computed sets $X_1$ and $X_2$ satisfying~\eqref{eq:sandwich2}. The intersection $X_1 \cap X_2$ is the desired set $X$ satisfying~\eqref{eq:sandwichth} for $\epsilon=0.018$ if $y=0.5$ or $y=1$, and $\epsilon=0.027$ if $y=1.5$. It is worth mentioning that the computed set $X$ is much smaller than the set $\epsilon\NE(f,\bbx)$ in all three cases. This can be seen in detail in Figure~\ref{fig:ex2alltogether} for the case $y=1.5$.
\begin{figure}[h]
		\centering
		\subfloat[][]{\includegraphics[width=0.35\linewidth]{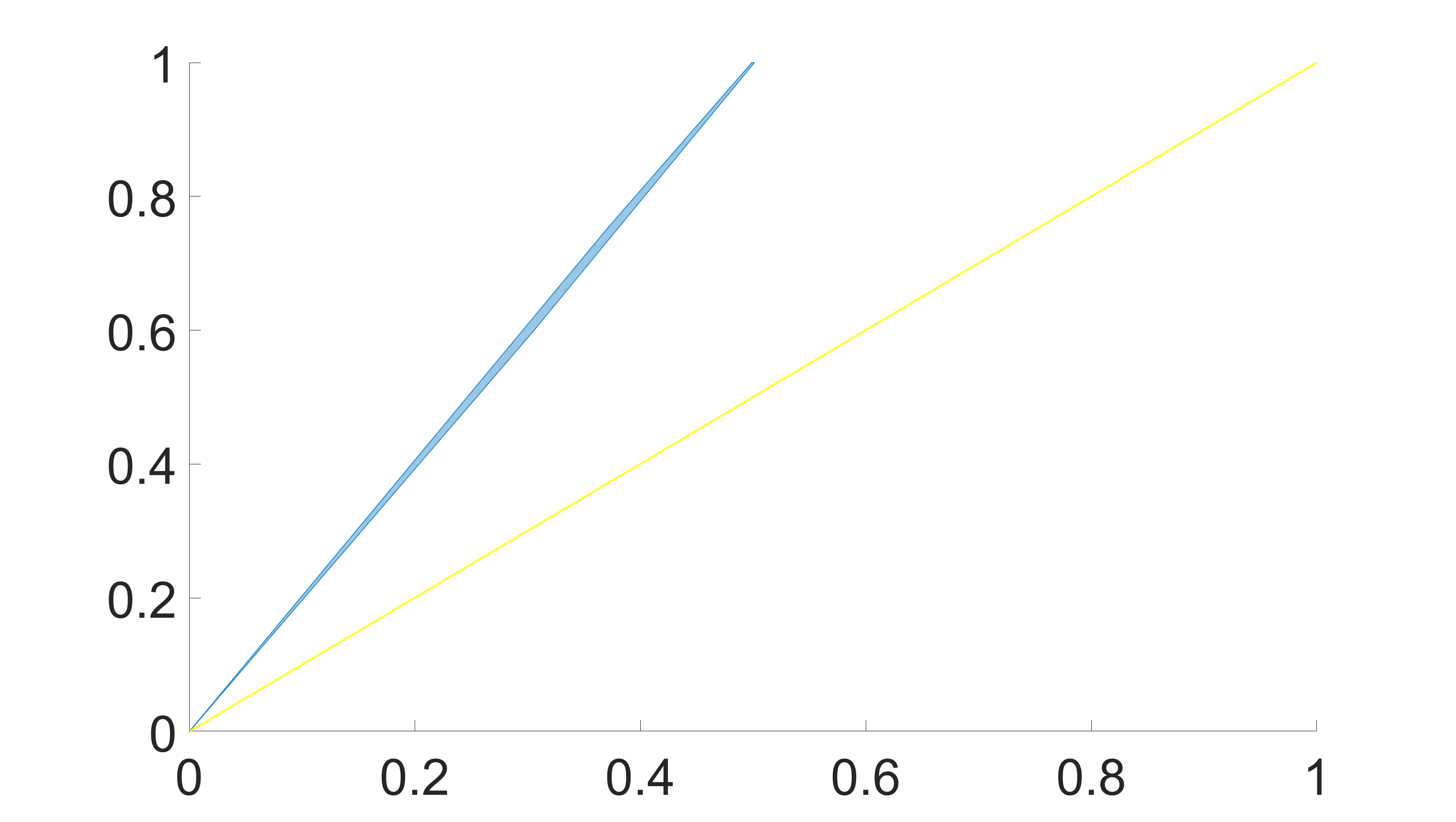}}
		\qquad
		\subfloat[][]{\includegraphics[width=0.35\linewidth]{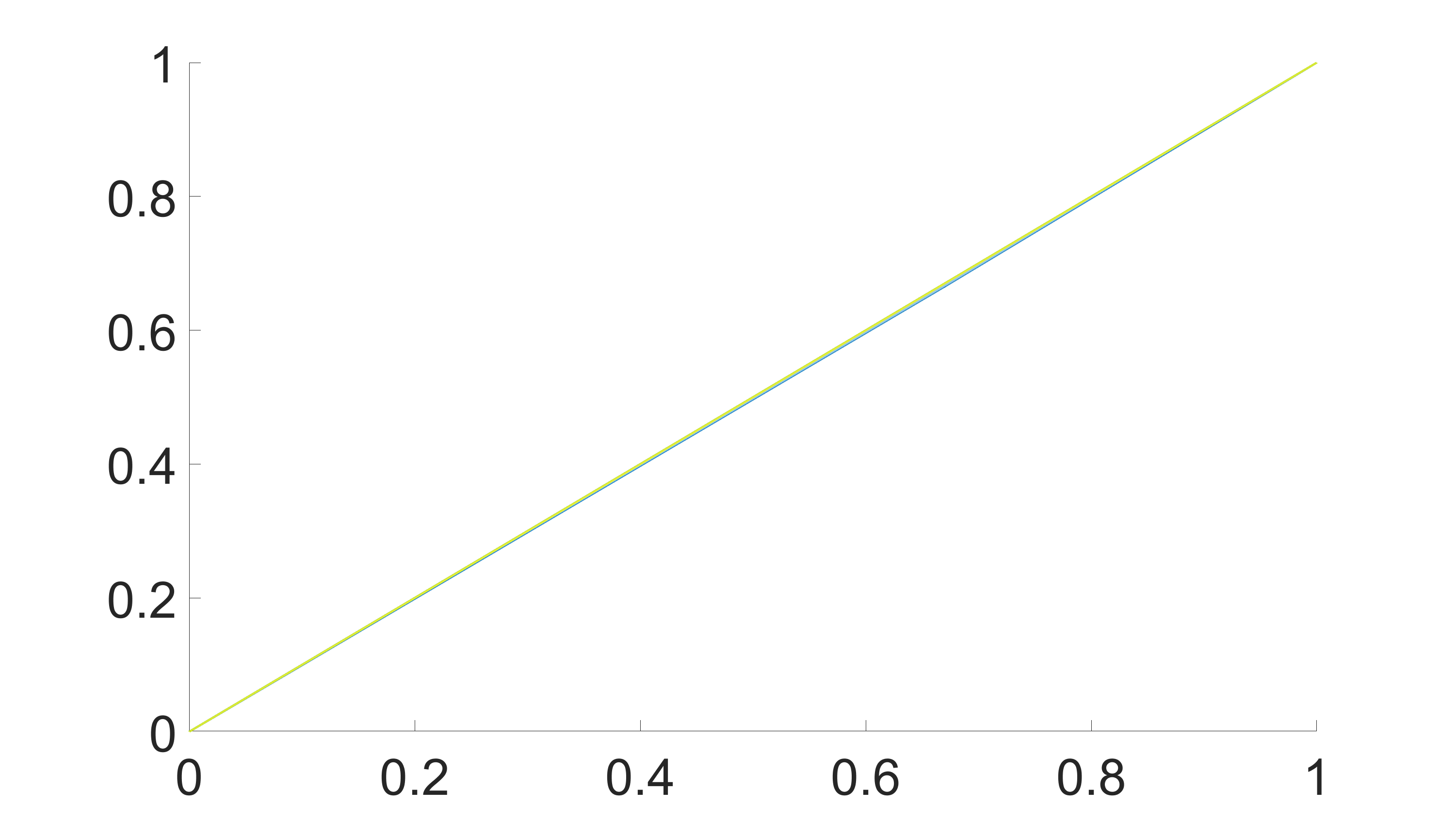}}
		\qquad
		\subfloat[][]{\includegraphics[width=0.35\linewidth]{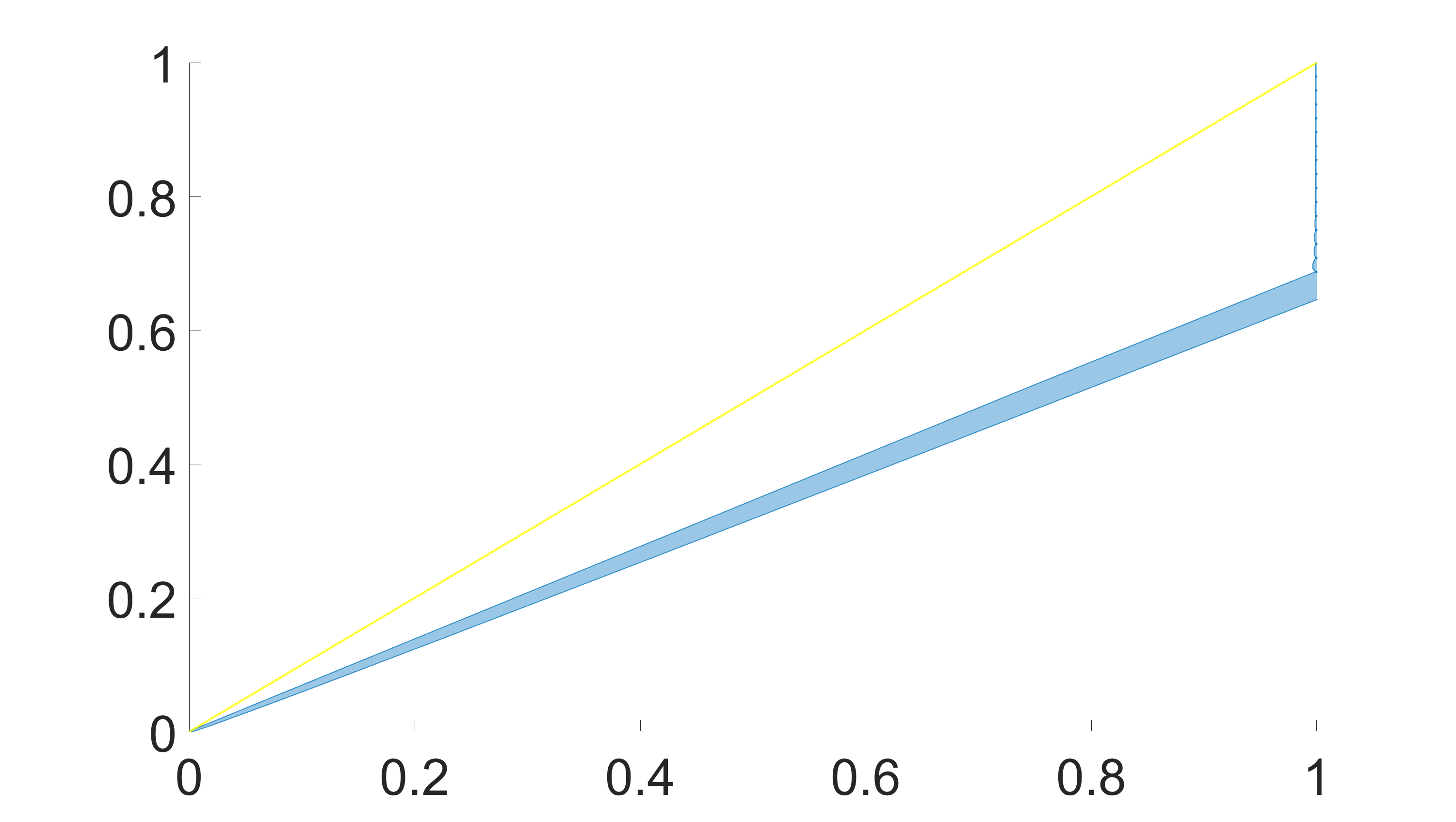}}
		\caption{Computed sets $X_1$ for player $1$ (blue) and $X_2$ for player $2$ (yellow) for (a) $y=0.5$, (b) $y=1$ and (c) $y=1.5$ in Example~\ref{ex:2}.} {\label{figure2}}
	\end{figure}

\begin{figure}[h]
		\centering
		   \includegraphics[width=0.4\linewidth]{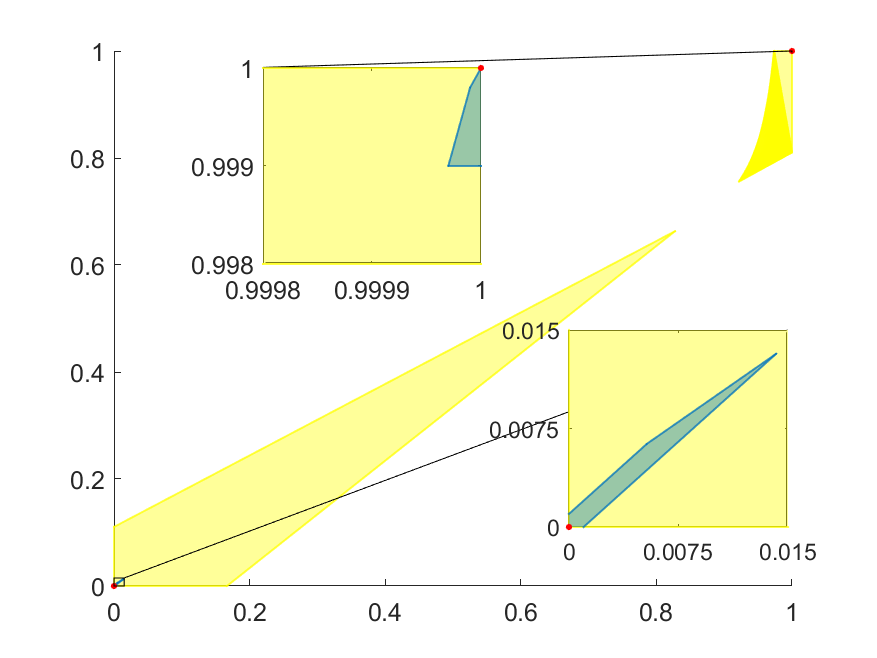}
		\caption{Example~\ref{ex:2}: the set $\epsilon\NE(f,\bbx)$ (yellow) as well as the computed set $X$ (blue) around the zoomed in regions of the true equilibria $(0,0)$ and $(1,1)$.} {\label{fig:ex2alltogether}}
	\end{figure}

\end{example}

\begin{example}
\label{ex:3}
Consider the example provided within Figure 2 of \cite{rosen65}. For this two player game, the convex cost functions are given by $f_1(x_1,x_2)=\frac{1}{2}x^2_1-x_1 x_2+x^2_2$ and $f_2(x_1,x_2)=x^2_2+x_1 x_2+x^2_1$ for $x_1,x_2\in\R$. A polyhedral shared constraint set is given by $\bbx=\{(x_1,x_2) \in \mathbb R^2\; | \; x_1+x_2 \geq 1, x_1 \geq 0, x_2 \geq 0\}$.  Note that the set $\bbx$ is not compact. However, a non-unique set of Nash equilibria for this game exists and is of the form $\NE(f,\bbx)=\{(x,1-x)\; | \;x \in [\frac{1}{2},1] \}$. 

In order to apply Algorithm~\ref{alg1}, we consider the compact subset $\bar{\bbx}=\bbx \cap [0,2]^2$. Since $\NE(f,\bbx) \subseteq \bar{\bbx}$, there is no loss of information by using this smaller constraint set for the computations. The Lipschitz constant of the joint cost function $(f_1,f_2)$ over $\Bar{\bbx}$ is $L=8$. We apply Algorithm~\ref{alg1} for chosen error levels $\varepsilon_1=0.01, \varepsilon_2=0.001$. Thus, the outputted set $X \subseteq \mathbb R^2$ satisfies~\eqref{eq:sandwichth} for $\epsilon=0.027$. Figure~\ref{fig:figure4} (a) shows the computed sets $X_1$ and $X_2$ satisfying~\eqref{eq:sandwich2} for both players and its intersection $X$. 
The exact set of $\epsilon$-approximate Nash equilibria for this example is given by
\begin{align*}
	\epsilon\NE(f,\bbx)&=\{(x_1,x_2) \in \R^2_+ \; | \; x_1 + x_2 \geq 1 , \; x_1 \in [x_2 - \sqrt{x_2^2 + 2(v_1(x_2) + \epsilon)}, \\
	&\quad x_2 + \sqrt{x_2^2 + 2(v_1(x_2) + \epsilon)}] , \; x_2 \in [1 - x_1 , \frac{1}{2}(-x_1 + \sqrt{x_1^2 + 4(v_2(x_1) + \epsilon)}\}
	\end{align*}
	 for $v_1(x_2) = \frac{1}{2}(x_2 \vee (1-x_2))^2 - (x_2 \vee (1-x_2))x_2$ and $v_2(x_1) = (1-x_1)^+$. Figure~\ref{fig:figure4} (b) shows the three sets of the sandwich result $\NE(f,\bbx) \subseteq X \subseteq \epsilon\NE(f,\bbx)$. Similar to before, the computed set $X$ is much smaller than the set $\epsilon\NE(f,\bbx)$ and approximates the set of true Nash equilibria $\NE(f,\bbx)$ quite well.

\begin{figure}[h]
		\centering
		\subfloat[][]{\includegraphics[width=0.4\linewidth]{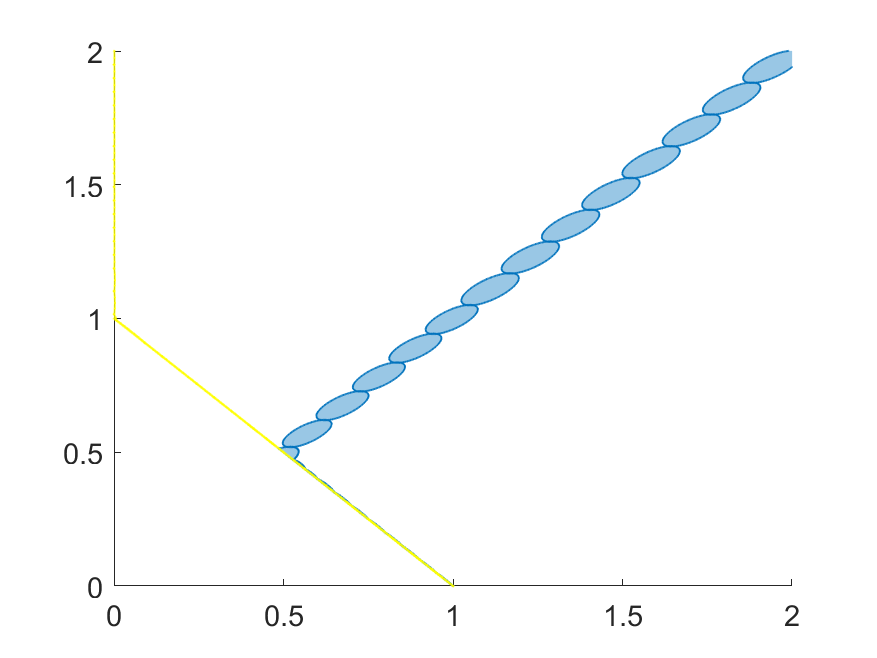}}
		\qquad
		\subfloat[][]{\includegraphics[width=0.4\linewidth]{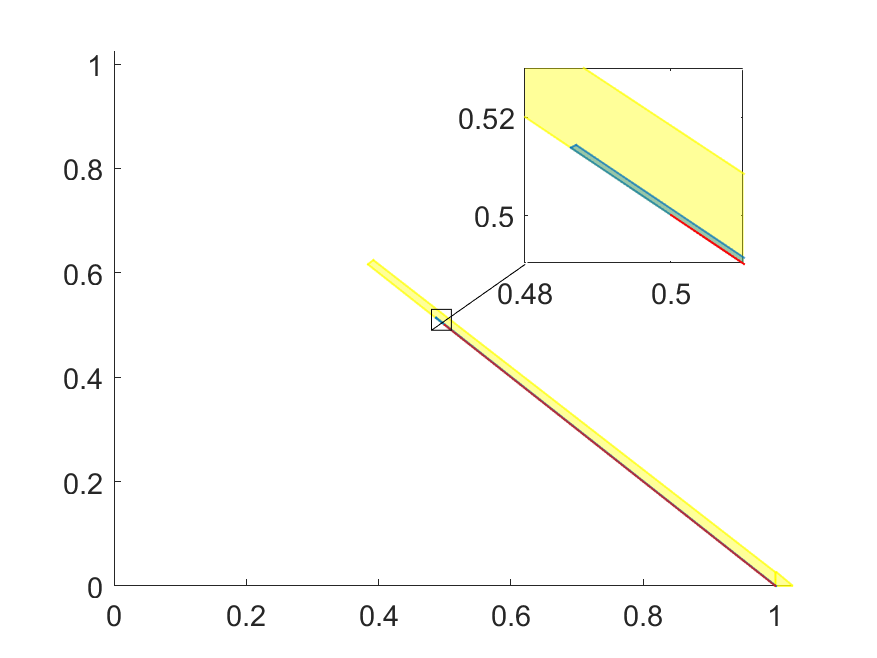}}
		\qquad
		\caption{Example~\ref{ex:3}: (a) Computed sets $X_1$ for player 1 (blue) and $X_2$ for player 2 (yellow). (b):  the set $\epsilon\NE(f,\bbx)$ (yellow), the computed set $X$ (blue) and the set $\NE(f,\bbx)$ (red) as well as a zoomed in region. } {\label{fig:figure4}}
	\end{figure}

\end{example}

\subsection{Search model\label{sec:search}}
Consider the $N$-player search model from~\cite{diamond1981mobility} and as described in~\cite[Section 3.1]{vives2005complementarities}.
In this game, player $i$ spends effort $x_i$ (chosen from the space $\xcal_i = \R$) to find a partner for a transaction; each player is constrained so as to only allow for nonnegative efforts $\bbx \subseteq \R^N_+$. The cost of this effort is provided by a smooth and strictly convex function $C: \R_+ \to \R_+$ that depends only on the action of the player, but is homogenous for all players. The expected payout is proportional to the effort put in by the player and some function $g: \R_+ \to \R_+$ of the aggregate effort of all other players.
Each player seeks to minimize her net costs $f_i(x) := C(x_i) - x_i g(\sum_{j \neq i} x_j)$.
\begin{remark}\label{rem:search}
In~\cite[Section 3.1]{vives2005complementarities}, no constraints are imposed on this game. Herein, so as to satisfy Assumption~\ref{ass:convex}, we impose the constraint set $\bbx = [0,1]^N$. In addition, to guarantee joint convexity of the cost function, we consider the modified cost functions $\hat f_i(x) = f_i(x) + \beta \sum_{j \neq i} x_j^2$ with $\beta > 0$ large enough.
\end{remark}

\begin{example}\label{ex:search}
Consider the case of two players. Assume a cubic cost function $C(z) := z^3$ and linear payout $g(z) := 0.5z$. It is trivial to verify that there exist two Nash equilibria for this problem $\{(0,0) \, , \, (1/6 , 1/6)\}$. To satisfy the conditions of Assumption~\ref{ass:convex} we include the two modifications to this original problem as outlined in Remark~\ref{rem:search} with $\beta = 2$. We note that the set of Nash equilibria is unimpacted by these modifications. 
The Lipschitz constant of the cost functions over $\bbx$ is $L = 7$. We apply Algorithm \ref{alg1} with error levels $\varepsilon_1 = \varepsilon_2 = 0.01$ to compute a set $X$ which satisfies~\eqref{eq:sandwichth} for $\epsilon = 0.15$. This set $X$, along with the Nash equilibria and $\epsilon$-Nash equilibria, are displayed in Figure~\ref{figure6}. 
We highlight that the computed set $X$ is comprised of distinct polyhedrons satisfying~\eqref{eq:sandwichth}.
\begin{figure}[h]
    \centering
    \includegraphics[width=0.4\linewidth]{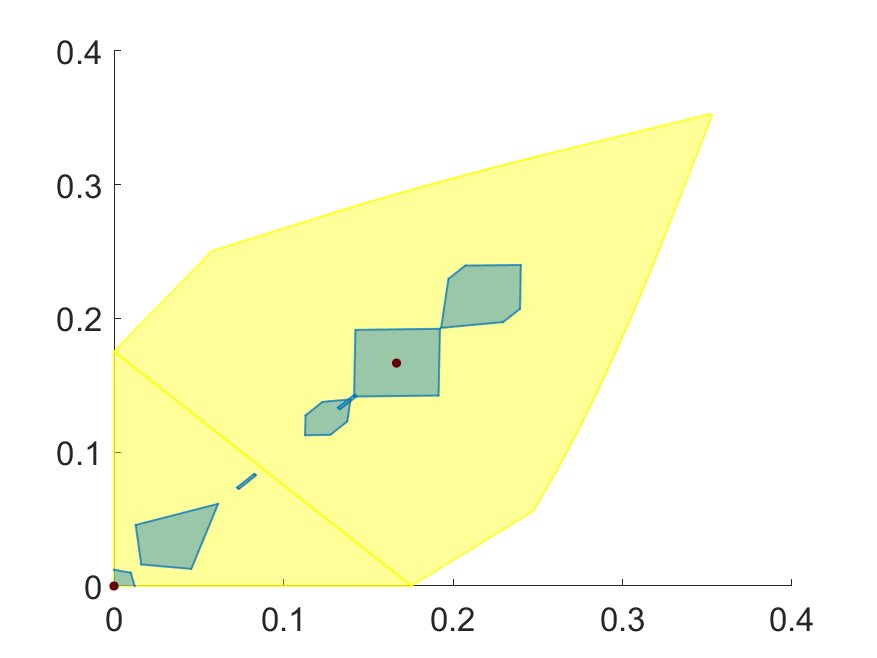}
    \caption{Example~\ref{ex:search}: the set $\epsilon\NE(f,\bbx)$ (yellow), the computed set $X$ (blue) and the set $\NE(f,\bbx)$ (red).}{\label{figure6}}
\end{figure}
\end{example}

\subsection{Pollution control game}\label{sec:pollution}
Let us consider the $N$-player pollution control game from environmental economics considered in~\cite{TZ2005}. There, each player represents a country that is, on one hand, seeking to maximing the net revenue from production of goods and services and, on the other hand, minimizing the environmental damage due to pollution. Pollution is assumed to be a proportional by-product of production and therefore the net revenues (gross revenue minus production cost) of each player, e.g., it can be expressed as a function of emissions. The strategy $x_i$ of country $i$, which is chosen from the space $\mathcal{X}_i = \R$, can be seen as the emissions of country $i$. Each country has its own revenue function $g_i: \mathcal{X}_i \to \mathbb R_+$ which is nonnegative and concave and depends only on the emissions of country $i$. However, the environmental damage due to emissions depends on the sum of all combined emissions which is stated by a convex damage cost function $d$. 
Each country tries to minimize its cost function $f_i$ (i.e.\ to maximize her total welfare given as the difference between the net revenue and the damage cost), 
 which is given by $f_i(x):=d(x_1+...+x_N)-g_i(x_i)$.
 Additionally, two types of constraints can be considered. Firstly, each country $i$ can face an environmental constraint $x_i \in [0,T_i]$ where $T_i \geq 0$ is an exogenously given upper bound on emissions, e.g.\ from an international treaty. Secondly, countries can agree to jointly fix an upper bound for emissions, which is stated by a polyhedral constraint of the form $\sum_{i=1}^N \alpha_i x_i \leq T$, where $\alpha_i \in \mathbb R_+$ for $i=1,...,N,$ are positive coefficients and $T>0$ is an upper bound. In~\cite{TZ2005}, the assumption of a unique Nash equilibrium is made and only those pollution control games are considered in their study. In contrast, with help of Theorem~\ref{thm:nash} (see also \cite[Th.~2.6]{FR23}) we can characterize the set of Nash equilibria as the set of all Pareto solutions of a certain vector optimization problem, regardless of uniqueness or finiteness of the set of equilibria. Furthermore, with Algorithm~\ref{alg1}, we are able to numerically approximate the set of Nash equilibria.
 In the following, we present examples for games with $2$ or $3$ players, and where the set of Nash equilibria is nontrivial.

\begin{example}
\label{ex:4}
Let us start with the case of two players. Assume each player has a linear revenue function given by $g_i(x_i)=\beta_i x_i$ with values $\beta_1=1.1$ and $\beta_2=2$. Consider a convex damage function $d(x)=\frac{1}{2}(x_1+x_2)^2$. Each player tries to minimize its cost $f_i(x)=d(x)-\beta_i x_i$. We assume that each player has its own box-constrained restrictions as well as a joint upper bound for the sum of their strategies (emissions). This leads to a compact polyhedral shared constraint set which is given by $\bbx=\{(x_1,x_2) \in [0,1]^2| x_1+0.4 x_2 \leq 1\}$. The Lipschitz constant of the cost functions over $\bbx$ is $L=2.1$. We apply Algorithm~\ref{alg1} with chosen error level $\epsilon_1=\epsilon_2=0.01$. Then the set $X$ computed by Algorithm~\ref{alg1} satisfies \eqref{eq:sandwich} for $\epsilon=0.052$. Figure~\ref{figure4} shows the computed sets $X_1$ and $X_2$ satisfying~\eqref{eq:sandwich2} for both players. The intersection is the set $X$. It contains a line segment as well as an isolated point area which comprise the set of Nash equilibria $\{(\frac{1}{10},1)\} \cup \{(x_1,\frac{5}{2}[1-x_1]) \; | \; x_1 \in [\frac{14}{15},1]\}$.

\begin{figure}[h]
		\centering
		{\includegraphics[width=0.4\linewidth]{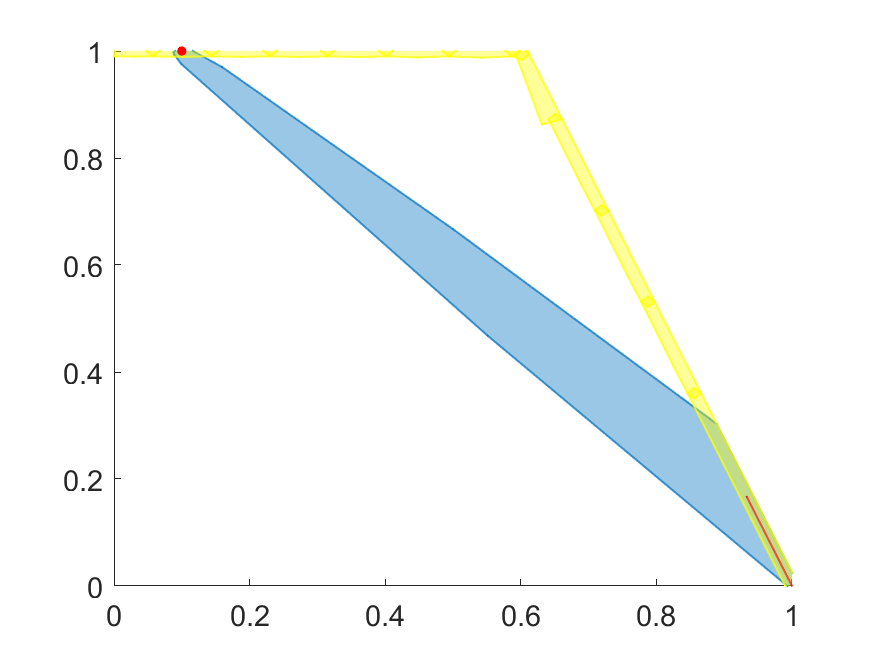}}
		\caption{Example~\ref{ex:4}: Computed sets $X_1$ for player 1 (blue) and $X_2$ for player 2 (yellow) as well as the true set $\NE(f,\bbx)$ (red). The intersection for both players yields the set $X$ that satisfies $\NE(f,\bbx) \subseteq X \subseteq \epsilon\NE(f,\bbx)$ for $\epsilon=0.052$.} {\label{figure4}}
	\end{figure}
\end{example}

\begin{example}
\label{ex:5}
We now extend this pollution game to include a third player. As in the two player case above, each player has a linear revenue function of the form $g_i(x_i)=\beta_i x_i$ where we set the values $\beta_1=1.1, \beta_2=1.3$ and $\beta_3=3.2$. Furthermore, we define a convex quadratic environmental damage function $d(x)=\frac{1}{2}(x_1+x_2+x_3)^2$. Each player tries to minimize its cost function $f_i(x)=d(x)-\beta_i x_i$.  Let the constraint set for this game be polyhedral with $\bbx=\{(x_1,x_2,x_3) \in [0,1]^3| x_1+0.6 x_2+0.4 x_3 \leq 1\}$. The Lipschitz constant of the cost over $\bbx$ takes the value $L=9.8$. We apply Algorithm~\ref{alg1} with error levels $\varepsilon_1=\varepsilon_2=0.01$ to compute a set $X$ which satisfies~\eqref{eq:sandwichth} for $\epsilon=0.1325$. This set $X$ is visualized in Figure~\ref{figure5}. Similar to the two player example we obtain a line segment unified with an isolated area.

\begin{figure}[h]
		\centering
		{\includegraphics[width=0.4\linewidth]{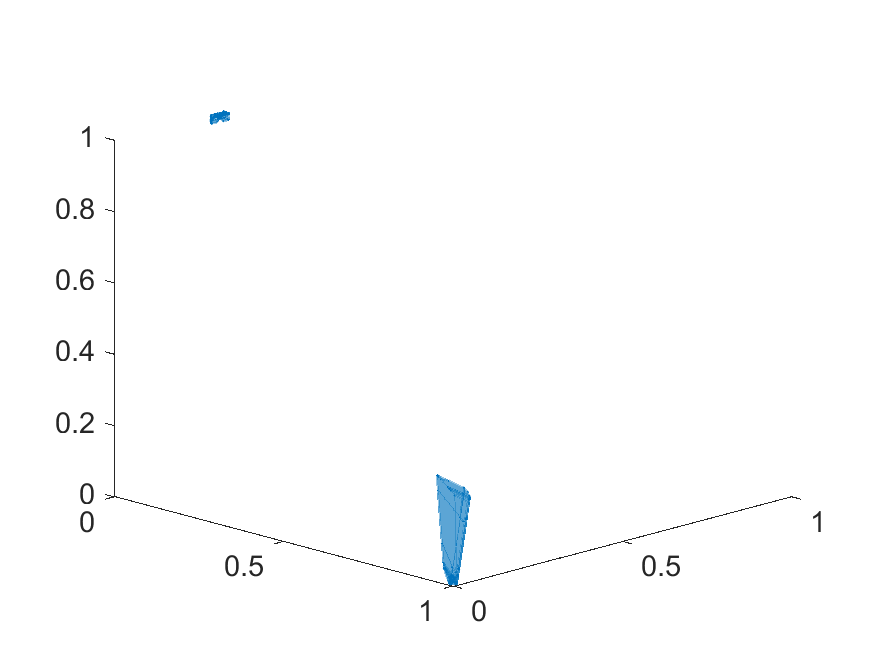}}
		\caption{Example~\ref{ex:5}: Set $X$ satisfying $\NE(f,\bbx) \subseteq X \subseteq \epsilon\NE(f,\bbx)$ for $ \epsilon=0.1325$.} {\label{figure5}}
	\end{figure}

\end{example}

\bibliographystyle{plain}
\bibliography{biblio-Nash}

\begin{thebibliography}{10}

\bibitem{Armand93}
Paul Armand.
\newblock Finding all maximal efficient faces in multiobjective linear
  programming.
\newblock {\em Mathematical Programming}, 61:357--375, 1993.

\bibitem{braouezec2021economic}
Yann Braouezec and Keyvan Kiani.
\newblock Economic foundations of generalized games with shared constraint: Do
  binding agreements lead to less {Nash} equilibria?
\newblock {\em European Journal of Operational Research}, 308(1):467--47, 2023.

\bibitem{diamond1981mobility}
Peter~A Diamond.
\newblock Mobility costs, frictional unemployment, and efficiency.
\newblock {\em Journal of political Economy}, 89(4):798--812, 1981.

\bibitem{facchinei2007generalized}
Francisco Facchinei, Andreas Fischer, and Veronica Piccialli.
\newblock On generalized nash games and variational inequalities.
\newblock {\em Operations Research Letters}, 35(2):159--164, 2007.

\bibitem{feinstein2022}
Zachary Feinstein.
\newblock Continuity and sensitivity analysis of parameterized nash games.
\newblock {\em Economic Theory Bulletin}, 10(2):233--249, 2022.

\bibitem{FR23}
Zachary Feinstein and Birgit Rudloff.
\newblock Characterizing and computing the set of {Nash} equilibria via vector
  optimization.
\newblock {\em Operations Research}, Forthcoming, 2023.

\bibitem{GJN06}
C{\'e}sar Guti{\'e}rrez, Bienvenido Jim{\'e}nez, and Vicente Novo.
\newblock On approximate solutions in vector optimization problems via
  scalarization.
\newblock {\em Computational Optimization and Applications}, 35(3):305--324,
  2006.

\bibitem{JRT2012}
Matthew~O. Jackson, Tomas Rodriguez-Barraquer, and Xu~Tan.
\newblock Epsilon-equilibria of perturbed games.
\newblock {\em Games and Economic Behavior}, 75(1):198--216, 2012.

\bibitem{Jahn11}
Johannes Jahn.
\newblock {\em Vector Optimization - Theory, Applications, and Extensions}.
\newblock Springer Science + Business Media, Berlin Heidelberg, second edition,
  2011.

\bibitem{kr22}
Gabriela {Kov{\'a}{\v{c}}ov{\'a}} and Birgit {Rudloff}.
\newblock Convex projection and convex multi-objective optimization.
\newblock {\em Journal of Global Optimization}, 83:301--327, 2022.

\bibitem{kr23}
Gabriela {Kov{\'a}{\v{c}}ov{\'a}} and Birgit {Rudloff}.
\newblock Approximations of unbounded convex projections and unbounded convex
  sets.
\newblock {\em Working Paper}, 2023.

\bibitem{Kulkarni2017}
Ankur~A. Kulkarni.
\newblock Games and teams with shared constraints.
\newblock {\em Philosophical Transactions of the Royal Society}, 375(2100),
  2017.

\bibitem{K79}
Semen~Samsonovich Kutateladze.
\newblock Convex $\epsilon$-programming.
\newblock {\em Soviet Math. Dokl}, 20(2):391--393, 1979.

\bibitem{LRU14}
Andreas L{\"o}hne, Birgit Rudloff, and Firdevs Ulus.
\newblock Primal and dual approximation algorithms for convex vector
  optimization problems.
\newblock {\em Journal of Global Optimization}, 60(4):713--736, 2014.

\bibitem{LZS21}
Andreas L{\"o}hne, Fangyuan Zhao, and Lizhen Shao.
\newblock On the approximation error for approximating convex bodies using
  multiobjective optimization.
\newblock {\em Applied Set-Valued Analysis and Optimization}, 3(3):341--354,
  2021.

\bibitem{nabetani2011parametrized}
Koichi Nabetani, Paul Tseng, and Masao Fukushima.
\newblock Parametrized variational inequality approaches to generalized nash
  equilibrium problems with shared constraints.
\newblock {\em Computational Optimization and Applications}, 48(3):423--452,
  2011.

\bibitem{nash1950}
John Nash.
\newblock Equilibrium points in $n$-person games.
\newblock {\em Proceedings of the National Academy of Sciences}, 36(1):48--49,
  1950.

\bibitem{nash1951}
John Nash.
\newblock Non-cooperative games.
\newblock {\em Annals of Mathematics}, pages 286--295, 1951.

\bibitem{nikaido1955note}
Hukukane Nikaid{\^o} and Kazuo Isoda.
\newblock Note on non-cooperative convex game.
\newblock {\em Pacific Journal of Mathematics}, 5(5):807--815, 1955.

\bibitem{nisan2007algorithmic}
Noam Nisan, Tim Roughgarden, \'{E}va Tardos, and Vijay~V. Vazirani.
\newblock {\em Algorithmic Game Theory}.
\newblock Cambridge University Press, 2007.

\bibitem{rosen65}
Judah~Ben Rosen.
\newblock Existence and uniqueness of equilibrium points for concave n-person
  games.
\newblock {\em Econometrica}, 33(3):520--534, 1965.

\bibitem{SZC18}
Lizhen Shao, Fangyuan Zhao, and Yuhao Cong.
\newblock Approximation of convex bodies by multiple objective optimization and
  an application in reachable sets.
\newblock {\em Optimization}, 67(6):783--796, 2018.

\bibitem{TZ2005}
Mabel Tidball and Georges Zaccour.
\newblock An environmental game with coupling constraints.
\newblock {\em Environmental Modeling \& Assessment}, 10(2):153--158, 2005.

\bibitem{tohidi2018adjacency}
Ghasem Tohidi and Hamid Hassasi.
\newblock Adjacency-based local top-down search method for finding maximal
  efficient faces in multiple objective linear programming.
\newblock {\em Naval Research Logistics (NRL)}, 65(3):203--217, 2018.

\bibitem{VanTu17}
Ta~Van~Tu.
\newblock A new method for determining all maximal efficient faces in multiple
  objective linear programming.
\newblock {\em Acta Math Vietnam}, 42:1--25, 2017.

\bibitem{vives2005complementarities}
Xavier Vives.
\newblock Complementarities and games: New developments.
\newblock {\em Journal of Economic Literature}, 43(2):437--479, 2005.

\end{thebibliography}

\end{document}